\documentclass[leqno,11pt]{article}%
\usepackage{amsfonts}
\usepackage{amsmath}
\usepackage{amssymb}
\usepackage{palatino}
\usepackage[colorlinks]{hyperref}
\usepackage{graphicx}%
\providecommand{\U}[1]{\protect\rule{.1in}{.1in}}
\hypersetup{colorlinks=true, urlcolor=red, linkcolor=blue, citecolor=blue}
\newtheorem{theorem}{Theorem}[section]

\newtheorem{definition}{Definition}[section]

\newtheorem{remark}{Remark}[section]

\newenvironment{proof}[1][Proof]{\noindent\textbf{#1.} }{\ \rule{0.5em}{0.5em}}
\begin{document}

\title{Phase-Space Realization of Stochastic Multivalued Dynamics with Gyroscopic
Subdifferential Coupling \thanks{E-mails: ioana.ciotir@insa-rouen.fr (Ioana
Ciotir), grajdeanuandreea19@gmail.com (Andreea Negru\c{t}),
eduard.rotenstein@uaic.ro (Eduard Rotenstein)\newline$\sharp$~corresponding
author}}
\author{Ioana Ciotir$^{a}$, Andreea Negru\c{t}-Gr\u{a}jdeanu$^{b}$, Eduard
Rotenstein$^{b,c,\sharp}$\medskip\\$^{a}${\small Normandie University, INSA de Rouen Normandie, LMI (EA 3226 --
FR CNRS 3335),}\\{\small 76000 Rouen, France}\\$^{b}${\small Simion Stoilow Institute of Mathematics of the Romanian Academy,
Bucharest, Romania}\\$^{c}${\small Faculty of Mathematics, "Alexandru Ioan Cuza" University, 9
Carol I Blvd., Ia\c{s}i,} {\small Romania}}
\date{}
\maketitle

\begin{abstract}
We introduce a new framework for stochastic variational inequalities with
oblique reflections generated by non-symmetric matrix coefficients. The loss
of symmetry fundamentally changes the structure of the reflected dynamics:
besides the dissipative component, it produces a genuinely gyroscopic
component describing rotational effects that are invisible to the associated
energy forms. To capture this hidden component, we consider planar obstacle
problems and we lift the dynamics to an extended phase-space system, where the
dissipative and gyroscopic mechanisms are naturally separated. This leads to a
new notion of solution, for which we establish existence and uniqueness. The
resulting framework allows for genuinely rotated reflections, which can lead
to angle-controlled rotations.

\end{abstract}

\textbf{Keywords and phrases: }multivalued stochastic dynamics, oblique
reflection, sub\-di\-ffe\-ren\-ti\-al operators

\textbf{MSC2020 Subject Classification: }60H10, 60H30, 49K45

\section{A problem with both dissipative and gyroscopic components}

We start from the classical concept of forward Stochastic Variational
Inequality (SVI, for short) with oblique reflection%
\begin{equation}
\left\{
\begin{array}
[c]{l}%
dX_{t}+\Theta\left(  \theta_{t}\right)  \partial\varphi\left(  X_{t}\right)
\left(  dt\right)  \ni f\left(  t,X_{t}\right)  dt+g\left(  t,X_{t}\right)
dW_{t},\quad\quad t>0,\smallskip\\
X_{0}=x_{0},
\end{array}
\right.  \label{eq to study}%
\end{equation}
where $W$ is a standard Wiener process, defined on a complete probability
space and the quantity $\Theta\left(  \theta_{t}\right)  $ is acting on the
set of subgradients; the product $\Theta\left(  \theta_{t}\right)
\partial\varphi\left(  X\right)  $ is called the set of \textit{oblique
subgradients}. When introduced by Gassous, R\u{a}\c{s}canu, Rotenstein
\cite{Gassous/Rascanu/Rotenstein:12}, \cite{Gassous/Rascanu/Rotenstein:15},
the problem proved to be challenging due to the presence of this new term,
which imposes the use of some specific approaches because it preserves neither
the monotony of the subdifferential operator nor the Lipschitz property of the
involved matrix, denoted by $H(X)$ in their papers. However, until the present
study, the authors imposed along other natural assumptions, a symmetry
property for it. This last hypothesis helps when used along standard arguments
of boundedness and convergences for some sequences of penalizing solutions,
but proves to be unnatural if we have in mind possible examples and
applications. Our aim is to renounce at this condition, which is not
satisfied, for example, by classical rotation matrices, for $d\geq2$. One can
construct applications, involving rotation angle-control problems, which can
be studied by the Pontryagin principle or by Stochastic Differential Games and
min-max convex optimization problems. But when renouncing at the symmetry of
the perturbing term, we face a question, which proves to be the main
difficulty in proving the existence and uniqueness, in a suitable sense, of
the solution of Eq.(\ref{eq to study}): \textit{how strong weaker condition
can be?}

Novelties of the research:

\begin{itemize}
\item We weaken the initial assumptions on the perturbing term, which were
introduced in the first studies on this topic by Gassous, R\u{a}\c{s}canu and
Rotenstein (\cite{Gassous/Rascanu/Rotenstein:12},
\cite{Gassous/Rascanu/Rotenstein:15}). More precisely, we renounce at its
symmetric property, as one can see in Section \ref{Working setup}, where we
describe the working setup and the technical assumptions on the coefficients
of Eq.(\ref{eq to study}). This proves to be more than a standard ad-on to the
existing results. Do to the nature of the method, we work with rotation
matrices, for $d=2$ and obstacle problems, which are the main situation to
reflect the new ingredients we face on.

\item The particular geometrical properties of the perturbing term will prove
to be crucial in the failure of the classical approaches for the standard
understanding of the problem, with Cauchy sequences techniques, Fitzpatrick
function, changes of the topology on the frontier, weighted norms, quadratic
energies and so on. This is why we have to construct a different approach
(similar to the Hamiltonian/port-Hamiltonian structures) and create a lifting
of our formal problem to a Phase-Space system, which is not equivalent with
the problem, but a consistent extension for it. This can control, via the
Phase equation the skew-symmetric part, which is now unlinked from the
dissipative symmetric part. In fact, we need to introduce a novel concept of
solution for SVIs with oblique subgradients, under a non-symmetric perturbing operator.
\end{itemize}

Let imagine that we have, for simplicity, $d=2$, and%
\begin{equation}
\Theta\left(  \theta_{t}\right)  :=\left(
\begin{array}
[c]{cc}%
\cos\theta_{t} & -\sin\theta_{t}\medskip\\
\sin\theta_{t} & \cos\theta_{t}%
\end{array}
\right)  \label{rotation matrix example}%
\end{equation}
is the rotation matrix of angle $\theta=\cos\theta_{t}\in(-\pi,\pi]$, for
which, the condition given imposed to $H$ is equivalent to $\theta\in\left[
-\pi/2+\delta,\pi/2-\delta\right]  $, for some arbitrary fixed $\delta
\in\left(  0,\pi/2\right)  .$ We can also extend it to%
\begin{equation}
\tilde{\Theta}_{t}:=\alpha_{t}\Theta\left(  \theta_{t}\right)  ,\quad
\text{where }\left(  \alpha,\theta\right)  \in C_{b}^{1}\left(  \left[
0,T\right]  ;\mathbb{R}_{+}\times\mathbb{R}^{2\ast2}\right)  .
\label{hyp on teta}%
\end{equation}
and $\Theta$ is a rotation matrix given by (\ref{rotation matrix example}) and
the restrictions for the values of $\alpha,\theta$ adapt accordingly. We see
that, in fact, only the presence of $\Theta$ is the consistent part, the
scaling function $\alpha$ being able, by a suitable change of variable, to be
absorbed into the structure of the normal cone at the frontier.

As one can see in the classical methods used for SVIs and BSVIs driven by
subdifferential operators perturbed by Lipschitz symmetrical matrices $H$
satisfying a coercivity condition - see, e.g.
(\cite{Gassous/Rascanu/Rotenstein:12}, \cite{Gassous/Rascanu/Rotenstein:15}) -
when applying classical or c\`{a}dl\`{a}g It\^{o}'s formula for a suitable
functional energy, a term which inherits the monotonicity properties given by
the Moreau-Yosida approximation appears. Existing results state that this term
can be controlled for SVIs and BSVIs driven by subdifferential operators,
unperturbed or perturbed by Lipschitz symmetric matrices. However, if we
renounce at the symmetry property, independent of the way of applying
It\^{o}'s formula, or, more general, even if we change the technique, a new
term, of the type $\left\langle \Delta U,\Theta\Delta Y\right\rangle
:=\left\langle U^{1}-U^{2},\Theta\left(  Y^{1}-Y^{2}\right)  \right\rangle $
needs to be controlled. This form can change if we change the approach for
these problems, but, beneath, the same conceptual / structural problem
appears. The difficulties come from the fact that this term also "sents" the
skew-symmetric part of $\Theta$. The symmetric part is the one producing
energy, the skew-symmetric part does not create or consume energy, it only
moves it inside the system, being the one destroying the estimates and here it
is the conceptual key for solving the problem.

The main question is, therefore, not \textit{what we need the close the
proof?}, but \textit{what property or decomposition in the analysis of
}$\Theta$\textit{ is really indispensable in order to sent the model? }A
structural different approach is mandatory. More precisely, the rotation
effect should not be grater than the dissipative one, or, it should be
controlled by a different dynamic. The perturbing term can be written, for
$d=2$, as:%
\[
\Theta\left(  \theta_{t}\right)  =\alpha_{t}\cos\theta_{t}\cdot I_{2}%
+\alpha_{t}\sin\theta_{t}\cdot J,\quad\text{where}\quad J=\left(
\begin{array}
[c]{cc}%
0 & -1\\
1 & 0
\end{array}
\right)  .
\]
The symmetric part produces the mechanical power $P=\left\langle \alpha
\cos\theta v,v\right\rangle =\alpha\cos\theta\left\vert v\right\vert ^{2}>0,$
for $\cos\theta>0$, so the system presents a dissipative effect, that is, the
energy decreases. For $d=2,$ in order to have the coercivity of the symmetric
part of $H_{t},$ $S_{t}=\alpha_{t}\cos\theta_{t}I,$ we need%
\[
\left\langle S_{t}u,u\right\rangle =\alpha_{t}\cos\theta_{t}\left\vert
u\right\vert ^{2}\geq s_{0}\left\vert u\right\vert ^{2},\quad a.e.~t\in\left[
0,T\right]  .
\]
So, we impose $\alpha_{t}\cos\theta_{t}\geq s_{0}>0.$ If, for example,
$\bar{\alpha}\geq\alpha_{t}\geq\underline{\alpha}>0$ and $\left\vert
\theta_{t}\right\vert \leq\theta_{\ast}<\pi/2$, then we can consider
$s_{0}:=\underline{\alpha}\cos\theta_{\ast}$~, condition which is acceptable
and it's stronger then the boundedness of $S_{t}$.

Since we remain devoted to It\^{o}'s type calculus, when applying the
Energy-type estimate for the approximated solution, we face the need to
incorporate the structure of the frontier of the reflection domain in order to
be able to manipulate the term featuring the trace operator, which should
inherit from somewhere higher order regularity. More precisely, in Section
\ref{Geom compatibility}, we introduce two general classes of domains which
allow us to deal with the second order derivative term. We give, in details,
the precise role of each class in developing the technique for the term which
is concerned.

\begin{remark}
The skew-symmetric part will give $P=\left\langle \alpha\sin\theta
Jv,v\right\rangle =0,$ since $Jv\perp v$. This force is not producing energy,
it only changes the rotation direction (e.g., the Coriolis force, the Lorentz
force in a magnetic field, the velocity in a fluid with vorticity), the energy
being redistributed. A similar discussion appears when $d=3$. The set of
parameter functions is now given by:%
\[
\left(  \alpha_{t},n_{t},\theta_{t}\right)  \in\lbrack\alpha_{\min}%
,\alpha_{\max}]\times\mathbb{S}^{2}\left(  0,1\right)  \times\left[
0,\theta_{\max}\right]  ,\quad\forall t\in\left[  0,T\right]  ,
\]
where $\mathbb{S}^{2}\left(  0,1\right)  :=\{n\in\mathbb{R}^{3}:\left\Vert
n\right\Vert =1\}$ and $\Theta\left(  n,\theta\right)  $ is the Rodrigues
rotation matrix corresponding the the rotation angle $\theta$ and rotation
axis $n:$%
\[
\Theta\left(  n,\theta\right)  :=I_{3}+\sin\theta\cdot R+\left(  1-\cos
\theta\right)  \cdot R^{2}~,
\]
with%
\[
R=\left(
\begin{array}
[c]{ccc}%
0 & -n_{3} & n_{2}\smallskip\\
n_{3} & 0 & -n_{1}\smallskip\\
-n_{2} & n_{1} & 0
\end{array}
\right)  .
\]
The main difficulty is caused by the fact that the defined energies don't
treat the skew-symmetric part, they do not feel it.
\end{remark}

\section{The working setup. Assumptions, hypothesis\label{Working setup}}

Let $\left(  \Omega,\mathcal{F},\mathbb{P},\mathbb{F}=\{\mathcal{F}%
_{t}\}_{t\geq0}\right)  $ be a stochastic basis and $\left\{  W_{t}%
:t\geq0\right\}  $ a $\mathbb{R}^{k}-$valued Wiener process. Our objective is
to solve the following SVI with oblique reflection, driven by convex
subdifferential operator (\ref{eq to study}). Suppose that we work under the
following assumptions:

\begin{itemize}
\item[$\left(  H_{f,g}\right)  :$] $\left(  t,x\right)  \longmapsto f\left(
t,x\right)  :\mathbb{R}_{+}\times\mathbb{R}^{2}\rightarrow\mathbb{R}^{2}%
$\ and$\;\left(  t,x\right)  \longmapsto g\left(  t,x\right)  :\mathbb{R}%
_{+}\times\mathbb{R}^{2}\rightarrow\mathbb{R}^{2\times k}\;$are
Carath\'{e}odory functions (i.e. measurable w.r. to $t$ and continuous w.r. to
$x$), satisfying%
\begin{equation}%
{\displaystyle\int_{0}^{T}}
(f^{\#}\left(  t\right)  )^{2}dt+%
{\displaystyle\int_{0}^{T}}
(g^{\#}\left(  t\right)  )^{4}dt<+\infty, \label{conv_ob-h6}%
\end{equation}
with%
\[
f^{\#}\left(  t\right)  :=\sup_{x\in Dom\left(  \varphi\right)  }\left\vert
f\left(  t,x\right)  \right\vert \quad\text{and}\quad g^{\#}\left(  t\right)
:=\sup_{x\in Dom\left(  \varphi\right)  }\left\vert g\left(  t,x\right)
\right\vert .
\]
We also add Lipschitz continuity conditions: there exist$~\mu\in L_{loc}%
^{1}\left(  \mathbb{R}_{+}\right)  ,$ $\ell\in L_{loc}^{2}\left(
\mathbb{R}_{+}\right)  $ such that, for every $x,y\in\mathbb{R}^{2}$,
$a.e.\ t\geq0$,%
\begin{equation}
\left\{
\begin{array}
[c]{rl}%
\left(  i\right)  \quad & \left\vert f\left(  t,x\right)  -f\left(
t,y\right)  \right\vert \leq\mu\left(  t\right)  \left\vert x-y\right\vert
,\medskip\\
\left(  ii\right)  \quad & \left\vert g\left(  t,x\right)  -g\left(
t,y\right)  \right\vert \leq\ell\left(  t\right)  \left\vert x-y\right\vert .
\end{array}
\right.  \label{conv_ob-h7}%
\end{equation}

\item[$\left(  H_{\varphi}\right)  :$] $\varphi:\mathbb{R}^{2}\rightarrow
(-\infty,+\infty]$\ is a proper convex lower semicontinuous (l.s.c.) function,
with $\partial\varphi$ being the subdifferential operator of $\varphi$:
\[
\partial\varphi\left(  x\right)  :=\left\{  \hat{x}\in\mathbb{R}%
^{2}:\left\langle \hat{x},y-x\right\rangle +\varphi\left(  x\right)
\leq\varphi\left(  y\right)  ,\text{ for all }y\in\mathbb{R}^{2}\right\}  ,
\]
and $Dom(\partial\varphi)=\{x\in\mathbb{R}^{2}:\partial\varphi(x)\neq
\emptyset\}$. We will use the notation $(x,\hat{x})\in\partial\varphi$ in
order to express that $x\in Dom(\partial\varphi)$ and $\hat{x}\in
\partial\varphi(x)$.\medskip

\item[$\left(  H_{\Theta}\right)  :$] the matrix $\Theta$ given by
(\ref{rotation matrix example}) obeys (\ref{hyp on teta}) and, for all
$t\in\left[  0,T\right]  $,%
\begin{equation}
\dfrac{1}{c}\left\vert u\right\vert ^{2}\leq\left\langle \Theta\left(
\theta_{t}\right)  u,u\right\rangle \leq c\left\vert u\right\vert ^{2}%
,\quad\forall u\in\mathbb{R}^{2}\quad\text{(for some\ }c\geq1\text{).}
\label{osp-h0-A}%
\end{equation}
Let $\left[  \Theta\left(  \theta_{t}\right)  \right]  ^{-1}=\left[
\Theta\left(  -\theta_{t}\right)  \right]  $ be the inverse matrix of
$\Theta\left(  \theta_{t}\right)  $ Assume that it has has the same properties
(\ref{osp-h0-A}) as $\Theta\left(  \theta_{t}\right)  $. The vector defined by
the quantity $\Theta\left(  \theta_{t}\right)  h$, with $h\in\partial
\varphi\left(  x\right)  $, will be called \textit{oblique subgradient}. In
fact, we have%
\[
\left\vert \left\vert \Theta\left(  \theta_{t}\right)  \right\vert \right\vert
_{L^{\infty}}+\left\vert \left\vert \Theta\left(  -\theta_{t}\right)
\right\vert \right\vert _{L^{\infty}}+\left\vert \left\vert S\right\vert
\right\vert _{L^{\infty}}+\left\vert \left\vert G\right\vert \right\vert
_{L^{\infty}}<+\infty,
\]
where $S,G$ are the symmetric, respectively, the skew-symmetric parts of
$\Theta\left(  \theta_{t}\right)  $.

\item[$\left(  H_{x_{0}}\right)  :$] $x_{0}\in Dom\left(  \varphi\right)
:=\{x\in\mathbb{R}^{2}:\varphi(x)<+\infty\}.$
\end{itemize}

\begin{remark}
If $E$ is a closed convex subset of $\mathbb{R}^{2}$, then the convexity
indicator function
\[
\varphi\left(  x\right)  =I_{E}\left(  x\right)  =\left\{
\begin{array}
[c]{cc}%
0, & \text{if }x\in E,\medskip\\
+\infty, & \text{if\ }x\notin E
\end{array}
\right.
\]
is a convex l.s.c. function and, for $x\in E$,
\[
\partial I_{E}\left(  x\right)  =\{\hat{x}\in\mathbb{R}^{2}:\left\langle
\hat{x},y-x\right\rangle \leq0,\;\forall~y\in E\}=N_{E}\left(  x\right)  ,
\]
where $N_{E}\left(  x\right)  $ is the closed external normal cone to $E$ at
$x$. We have $N_{E}\left(  x\right)  =\emptyset$ if $x\notin E$ and
$N_{E}\left(  x\right)  =\left\{  0\right\}  $ if $x\in int\left(  E\right)  $
(we denote by $int\left(  E\right)  $ the interior of the set $E$).
\end{remark}

\begin{remark}
Denote by $S_{2}^{p}\left[  0,T\right]  $, $p\geq0$, the space of
progressively measurable continuous stochastic processes $X:\Omega
\times\left[  0,T\right]  \rightarrow\mathbb{R}^{2}$, such that%
\[
\left\Vert X\right\Vert _{S_{2}^{p}}=\left\{
\begin{array}
[c]{ll}%
\left(  \mathbb{E}\left\Vert X\right\Vert _{T}^{p}\right)  ^{\frac{1}{p}%
\wedge1}<+{\infty}, & \;\text{if }p>0,\medskip\\
\mathbb{E}\left[  1\wedge\left\Vert X\right\Vert _{T}\right]  , & \;\text{if
}p=0,
\end{array}
\right.
\]
where $\left\Vert X\right\Vert _{T}=\sup_{t\in\left[  0,T\right]  }\left\vert
X_{t}\right\vert $. The space $(S_{2}^{p}\left[  0,T\right]  ,\left\Vert
\cdot\right\Vert _{S_{2}^{p}}),\ p\!\geq1,$ is a Banach space and $S_{2}%
^{p}\left[  0,T\right]  $, $0\leq p<1$, is a complete metric space with the
metric $\rho(Z_{1},Z_{2})=\left\Vert Z_{1}-Z_{2}\right\Vert _{S_{2}^{p}}$
(when $p=0$ the metric convergence coincides with the probability convergence).

Denote by $\Lambda_{2\times k}^{p}\left(  0,T\right)  ,\ p\in\lbrack
0,+\infty)$, the space of progressively measurable stochastic processes
$Z:{\Omega}\times(0,T)\rightarrow\mathbb{R}^{2\times k}$ such that%
\[
\left\Vert Z\right\Vert _{\Lambda^{p}}=\left\{
\begin{array}
[c]{ll}%
\left[  \mathbb{E}\left(
{\displaystyle\int_{0}^{T}}
\Vert Z_{s}\Vert^{2}ds\right)  ^{\frac{p}{2}}\right]  ^{\frac{1}{p}\wedge1}, &
\;\text{if }p>0,\bigskip\\
\mathbb{E}\left[  1\wedge\left(
{\displaystyle\int_{0}^{T}}
\Vert Z_{s}\Vert^{2}ds\right)  ^{\frac{1}{2}}\right]  , & \;\text{if }p=0.
\end{array}
\right.
\]
The space $(\Lambda_{2\times k}^{p}\left(  0,T\right)  ,\left\Vert
\cdot\right\Vert _{\Lambda^{p}}),\ p\geq1,$ is a Banach space and
$\Lambda_{2\times k}^{p}\left(  0,T\right)  $, $0\leq p<1,$ is a complete
metric space with the metric $\rho(Z_{1},Z_{2})=\left\Vert Z_{1}%
-Z_{2}\right\Vert _{\Lambda^{p}}$.
\end{remark}

\section{The lifting to a closed phase-space system. Main result}

We revisit the obstacle problem given by Eq.(\ref{eq to study}) and observe
that we can understand it as%
\[
dX_{t}+\alpha_{t}\Theta\left(  \theta_{t}\right)  dK_{t}=f\left(
t,X_{t}\right)  dt+g\left(  t,X_{t}\right)  dW_{t},\quad dK_{t}\in
\partial\varphi\left(  X_{t}\right)  \left(  dt\right)  ,\quad\mathbb{P-}%
\text{a.e.}%
\]
If we have in mind the second law of dynamics, the pure rotation brought by
$\Theta\left(  \theta_{t}\right)  $ can be decomposed as%
\[
\Theta\left(  \theta_{t}\right)  =\cos\theta_{t}I_{2}+\sin\theta_{t}%
J,\quad\text{where}\quad J^{T}=-J.
\]
$\theta$ is the control parameter of the rotation angle and $\alpha$ is the
scalar amplitude/intensity of the rotation. We denote%
\begin{equation}
\left\{
\begin{array}
[c]{cccl}%
S_{t} & := & \alpha_{t}\cos\theta_{t}I, & \quad\text{the symmetric part of
}\alpha_{t}\Theta\left(  \theta_{t}\right)  ;\medskip\\
G_{t} & := & \alpha_{t}\sin\theta_{t}J, & \quad\text{the skew-symmetric part
of }\alpha_{t}\Theta\left(  \theta_{t}\right)  .
\end{array}
\right.  \label{particular form d=2}%
\end{equation}
We make clear the distinct roles of the dissipative part and the gyroscopic
part. Consider $S$, respectively, $G$ the symmetric, respectively, the
skew-symmetric parts of $\Theta$, with their particular formulas, for $d=2$,
given by (\ref{particular form d=2}), but defined in the general case $d\geq2$
as:%
\begin{equation}
S:=\dfrac{1}{2}\left(  H+H^{T}\right)  \quad\quad\text{and}\quad\quad
G:=\dfrac{1}{2}\left(  H-H^{T}\right)  . \label{definition of the two parts}%
\end{equation}

Formally, the initial equation becomes%
\begin{equation}
\left\{
\begin{array}
[c]{l}%
dX_{t}+S_{t}dK_{t}+G_{t}dK_{t}=f\left(  t,X_{t}\right)  dt+g\left(
t,X_{t}\right)  dW_{t},\medskip\\
X_{0}=x_{0},\quad\text{and }dK_{t}\in\partial\varphi\left(  X_{t}\right)
\left(  dt\right)  ,\quad\mathbb{P-}\text{a.e,}%
\end{array}
\right.  \label{decomposed equation}%
\end{equation}
with the problematic term in the energy computations $\left\langle
X_{t}^{\varepsilon},G_{t}U_{t}^{\varepsilon}\right\rangle $ because it can not
be controlled by Young constants or by the Moreau-Yosida properties of the
approximating terms. In order to be able to incorporate in the analysis the
gyroscopic component the key is to understand the notion of solution in an
extended way, by proposing first a lifted closed phase-space system, as
follows:%
\begin{equation}
\left\{
\begin{array}
[c]{l}%
dX_{t}+S_{t}dK_{t}+G_{t}P_{t}dt=f\left(  t,X_{t}\right)  dt+g\left(
t,X_{t}\right)  dW_{t},\medskip\\
dP_{t}+G_{t}dK_{t}+\Gamma_{t}P_{t}dt=0,\medskip\\
dK_{t}\in\partial\varphi\left(  X_{t}\right)  \left(  dt\right)
,\quad\mathbb{P-}\text{a.e.,}\quad X_{0}=x_{0},\text{ }P_{0}=p_{0}.
\end{array}
\right.  \label{Phase-Space ssytem}%
\end{equation}
In the phase-space realization and throughout the article, we work with the
time-dependent gyroscopic coefficient; hence $S$ and $G$ are continuous
bounded matrix-valued functions of time, but the general result can be
extended with a space dependence also. This is beyond the aim of our present
study. The variable $P$ represents the gyroscopic effect into a phase space.
The presence of $\Gamma$ is not mandatory, but it is an additional cancelling
term in the phase dynamic. The term $\Gamma P$ should control the dissipation
of the phase energy. If it is of symmetric type and%
\[
\left\vert \left\vert \Gamma\right\vert \right\vert _{L^{\infty}\left(
0,T;\mathcal{L}\left(  \mathbb{R}^{2}\right)  \right)  }<+\infty
\]
when used together with the Moreau-Yosida penalization, its contribution in
the energetic balance should be damped. For instance, one can have
$\Gamma:=\gamma I,$ for a the damping coefficient $\gamma\geq0.$ If $\gamma
>0$, the phase is damped, for $\gamma=0,$ the phase is purely conservative.
More precisely, if we suppose that $\Gamma$ is strictly positively defined, we
have a coercive damping, of parameter $\gamma_{0}$, over all the directions,
while, if $\Gamma$ is only positively defined one can have damping only on
some directions. Usually, $\Gamma=0$ if our intention is to identify the
gyroscopic effect. Adding a nonzero $\Gamma$ is useful for stability,
dissipation behavior and control of the phase intensity. If it is the case, we
assume that $\Gamma$ is $\gamma_{0}-$coercive.

Also, the second equation from (\ref{Phase-Space ssytem}) is not necessary to
be of homogeneous type, it can borrow for the right hand side parts of $f$ and
$g,$ if necessary. It can gain a drift $f_{phase}$ or a diffusion $g_{phase}$,
depending on $t,X$ and $P.$ Usually,%
\begin{equation}
f_{phase}\left(  t,x,p\right)  :=k_{f}~f\left(  t,x\right)  ,\quad
\text{with}\quad k_{f}\in\mathbb{R} \label{choise of f phase}%
\end{equation}
or, forcing with feedback:%
\[
f_{phase}\left(  t,x,p\right)  :=k_{f}~f\left(  t,x\right)  -\Lambda p,
\]
where $-\Lambda p$ is an additional relaxing term, which can be combined with
$\Gamma p.$ For $g_{phase}$ it is preferable a form like
(\ref{choise of f phase}), with a constant $k_{g}.$ In this case, $f_{phase}$
and $g_{phase}$ must obey the same hypotheses $\left(  H_{f,g}\right)  $ as
$f$ and $g$, but with the additional variable $p.$ They must be Lipschitz also
in $p$.

We propose the following solution concept for the Eq.(\ref{eq to study}),
which should be interpreted correctly by its extension to the system
(\ref{Phase-Space ssytem}).

\begin{definition}
\label{Definition of Solution}\textbf{(I).} Given a stochastic basis
$(\Omega,\mathcal{F},\mathbb{P},\mathbb{F=}\left\{  \mathcal{F}_{t}\right\}
_{t\geq0})$ and a $\mathbb{R}^{k}-$valued $\mathcal{F}_{t}-$Wiener process
$W$, we say that a triplet $\left(  X,P,K\right)  :$ $\Omega\times
\lbrack0,T]\rightarrow\mathbb{R}^{3\ast2}$ of continuous $\mathcal{F}_{t}%
-$progressively measurable stochastic processes is a strong solution on the
time interval $\left[  0,T\right]  $ for the Stochastic Variational Inequality
with oblique reflection (\ref{Phase-Space ssytem}) if, $\mathbb{P}%
-$a.s.$\;\omega\in\Omega$, the following are satisfied:

\begin{enumerate}
\item The processes $X$ and $P$ are $\mathbb{F}-$adapted and $X_{t}%
\in\overline{Dom\left(  \varphi\right)  },$ $\forall\,t\geq0,\;\varphi\left(
X_{\cdot}\right)  \in L^{1}\left(  \left[  0,T\right]  \right)  .$

\item The feedback process $K$ is $\mathbb{F}-$adapted, with continuous
trajectories, $K_{0}=0$ and $K_{\cdot}\in BV\left(  \left[  0,T\right]
;\mathbb{R}^{2}\right)  .$ More precisely, $\mathbb{E}\left\updownarrow
K\right\updownarrow _{T}^{2}<+\infty.$ The notation $\left\updownarrow
K\right\updownarrow _{T}^{2}$ represents the total variation on the time
interval $\left[  0,T\right]  $ of the function $K$.

\item For every $t\in\left[  0,T\right]  ,$ $\mathbb{P}-$a.s.%
\[
\left\{
\begin{array}
[c]{l}%
X_{t}+%
{\displaystyle\int_{0}^{t}}
S_{s}dK_{s}+%
{\displaystyle\int_{0}^{t}}
G_{s}P_{s}ds=x_{0}+%
{\displaystyle\int_{0}^{t}}
f\left(  s,X_{s}\right)  ds+%
{\displaystyle\int_{0}^{t}}
g\left(  s,X_{s}\right)  dW_{s}\medskip\\
P_{t}+%
{\displaystyle\int_{0}^{t}}
G_{s}dK_{s}+%
{\displaystyle\int_{0}^{t}}
\Gamma_{s}P_{s}ds=p_{0}+%
{\displaystyle\int_{0}^{t}}
f_{phase}\left(  s,X_{s},P_{s}\right)  ds+%
{\displaystyle\int_{0}^{t}}
g_{phase}\left(  s,X_{s},P_{s}\right)  dW_{s}%
\end{array}
\right.
\]

\item The process $K$ satisfies $dK_{t}\in\partial\varphi\left(  X_{t}\right)
\left(  dt\right)  .$ More precisely, the vectorial measure $dK$ is associated
to the $Graph(\partial\varphi):$%
\[
\left(  X_{t},\frac{dK_{t}}{d\left\vert K\right\vert _{t}}\right)  \in
Graph(\partial\varphi),\quad d\left\vert K\right\vert -a.e.
\]

\end{enumerate}

\textbf{(II). }If there exist a new stochastic basis $\left(  \bar{\Omega
},\mathcal{\bar{F}},\mathbb{\bar{P}},\mathbb{\bar{F}}\right)  $ and a new
$\mathbb{\bar{F}}$-Wiener process $\bar{W}$, such that, with respect to
$\left(  \bar{\Omega},\mathcal{\bar{F}},\mathbb{\bar{P}},\mathbb{\bar{F}}%
,\bar{W}\right)  $, there exists a solution $\left(  \bar{X},\bar{P},\bar
{K}\right)  $ in the sense of \textbf{(I), }then the collection
\[
\left(  \bar{\Omega},\mathcal{\bar{F}},\mathbb{\bar{P}},\mathbb{\bar{F}}%
,\bar{W},\bar{X},\bar{P},\bar{K}\right)
\]
is called a weak-strong solution for (\ref{Phase-Space ssytem}). Remark that
the solution is weak from the probabilistic point of view (it's on a changed
space), but it is strong from the systems' perspective (it's not a variational
solution or given on test functions).
\end{definition}

\begin{remark}
Condition $\left(  4\right)  $ from Definition \ref{Definition of Solution} is
equivalent with%
\[%
{\displaystyle\int_{s}^{t}}
\left\langle y\left(  r\right)  -X_{r},dK_{r}\right\rangle +%
{\displaystyle\int_{s}^{t}}
\varphi\left(  X_{r}\right)  dr\leq%
{\displaystyle\int_{s}^{t}}
\varphi\left(  y\left(  r\right)  \right)  dr,\quad\mathbb{P}-a.s.,
\]
for every$\,0\leq s\leq t,\;$and $y\in C\left(  \left[  0,T\right]
;\mathbb{R}^{2}\right)  $.
\end{remark}

\begin{remark}
In particular, for the convexity indicator function $\varphi:\mathbb{R}%
^{2}\rightarrow(-\infty,+\infty],$ $\varphi=I_{C},$ where $C\subset
\mathbb{R}^{2}$ is a nonempty, closed, bounded, convex set, the above
definition states that $\frac{dK_{t}}{d\left\vert K\right\vert _{t}}\in
N_{C}\left(  X_{t}\right)  ,$ $d\left\vert K\right\vert -a.e.$ and the first
component of the solution, $X_{t}\in C,$ for every $t.$
\end{remark}

\begin{theorem}
\label{main result}Under all the assumed assumptions, the Stochastic
Variational Inequality with oblique reflection (\ref{Phase-Space ssytem})
admits at least one solution in the sense of Definition
\ref{Definition of Solution}. Moreover, $\left(  X,P\right)  \in S_{2}%
^{2}\left[  0,T\right]  \times S_{2}^{2}\left[  0,T\right]  :$%
\[
\mathbb{E}\left(  \sup\limits_{0\leq t\leq T}\left\vert X_{t}\right\vert
^{2}+\sup\limits_{0\leq t\leq T}\left\vert P_{t}\right\vert ^{2}\right)
<+\infty.
\]

\end{theorem}

\section{Proof of the main result}

The proof of Theorem \ref{main result} is divided into five steps, in order to
distinguish the important milestones.

\subsection{A penalizing system}

\begin{proof}
For every fixed $\varepsilon>0,$ consider the penalized phase-space system:%
\begin{equation}
\left(  P_{e}\right)  :\left\{
\begin{array}
[c]{l}%
dX_{t}^{\varepsilon}+S_{t}\nabla\varphi_{\varepsilon}\left(  X_{t}%
^{\varepsilon}\right)  dt+G_{t}P_{t}^{\varepsilon}dt=f\left(  t,X_{t}%
^{\varepsilon}\right)  dt+g\left(  t,X_{t}^{\varepsilon}\right)
dW_{t},\medskip\\
dP_{t}^{\varepsilon}+G_{t}\nabla\varphi_{\varepsilon}\left(  X_{t}%
^{\varepsilon}\right)  dt+\Gamma_{t}P_{t}^{\varepsilon}dt=k_{f}~f\left(
t,X_{t}^{\varepsilon},P_{t}^{\varepsilon}\right)  dt+k_{g}~g\left(
t,X_{t}^{\varepsilon},P_{t}^{\varepsilon}\right)  dW_{t},\medskip\\
U_{t}^{\varepsilon}:=\nabla\varphi_{\varepsilon}\left(  X_{t}^{\varepsilon
}\right)  ,\quad\mathbb{P-}\text{a.e.,}\quad X_{0}^{\varepsilon}=X_{0},\text{
}P_{0}^{\varepsilon}=P_{0}.
\end{array}
\right.  \label{penalized Phase-Space system}%
\end{equation}
Keeping in mind the properties of $dK_{t}^{\varepsilon}:=U_{t}^{\varepsilon
}dt=\nabla\varphi_{\varepsilon}\left(  X_{t}^{\varepsilon}\right)  dt$, when
passing to the limit as $\varepsilon\searrow0$, the entire term will converge
in the sense of a Riemann-Stieltjes integral. Therefore, the multivalued limit
term will be considered with respect to a measure $K.$ Let us first verify
that, for every $\varepsilon>0$, the system
(\ref{penalized Phase-Space system}) admits a unique solution. Denote%
\[
Z_{t}^{\varepsilon}:=\left(
\begin{array}
[c]{c}%
X_{t}^{\varepsilon}\smallskip\\
P_{t}^{\varepsilon}%
\end{array}
\right)  \in\mathbb{R}^{2\ast2}.
\]
Then, the system (\ref{penalized Phase-Space system}) can be rewritten as:%
\[
dZ_{t}^{\varepsilon}=B_{\varepsilon}\left(  t,Z_{t}^{\varepsilon}\right)
dt+\Sigma_{\varepsilon}\left(  t,Z_{t}^{\varepsilon}\right)  dW_{t}~,
\]
where, by reminding the formula of the Yosida approximation of $\varphi,$
$A_{\varepsilon}\left(  x\right)  =\nabla\varphi_{\varepsilon}\left(
x\right)  $,%
\[
B_{\varepsilon}\left(  t,x,p\right)  :=\left(
\begin{array}
[c]{c}%
f\left(  t,x\right)  -S_{t}A_{\varepsilon}\left(  x\right)  -G_{t}%
p\smallskip\\
k_{f}f\left(  t,x\right)  -G_{t}A_{\varepsilon}\left(  x\right)  -\Gamma_{t}p
\end{array}
\right)  \quad\text{and}\quad\Sigma_{\varepsilon}\left(  t,x,p\right)
:=\left(
\begin{array}
[c]{c}%
g\left(  t,x\right) \\
k_{g}~g\left(  t,x,p\right)
\end{array}
\right)  .
\]
The Yosida penalization $A_{\varepsilon}$ is an $1/\varepsilon-$Lipschitz
univalued operator. Also, there exist the Lipschitz constants $L_{f}%
,L_{g},L_{f_{phase}},L_{g_{phase}}$ for the drivers and diffusions from the
right hand sides (RHS) of the two equations of the system. According to
(\ref{definition of the two parts}),%
\[
\max\left\{  \left\vert \left\vert S_{t}\right\vert \right\vert _{\mathrm{op}%
},\left\vert \left\vert G_{t}\right\vert \right\vert _{\mathrm{op}}\right\}
\leq\left\vert \left\vert \alpha_{t}\Theta\left(  \theta_{t}\right)
\right\vert \right\vert _{\mathrm{op}}\leq M_{T}=:\operatorname*{ess}%
\sup_{0\leq t\leq T}\left\vert \left\vert \alpha_{t}\Theta\left(  \theta
_{t}\right)  \right\vert \right\vert _{\mathrm{op}}<+\infty.\footnote{We have
$\left\vert \left\vert \alpha_{t}\Theta\left(  \theta_{t}\right)  \right\vert
\right\vert _{\mathrm{op}}=\left\vert \alpha_{t}\right\vert .$ Hence
$M_{T}=\left\vert \left\vert \alpha\right\vert \right\vert _{L^{\infty}\left(
0,T\right)  }.$}%
\]
A similar assumption is imposed to $\Gamma$, $\sup_{t\in\left[  0,T\right]
}\left\vert \left\vert \Gamma_{t}\right\vert \right\vert <+\infty$.
Consequently, for the penalized drift (by using the generic notation $\Delta
h:=h\left(  x\right)  -h\left(  y\right)  $, for an arbitrary function $h$),%
\begin{align*}
\Delta B_{\varepsilon}\left(  t\right)   &  =B_{\varepsilon}\left(
t,x,p\right)  -B_{\varepsilon}\left(  t,y,q\right) \\
&  =\left(
\begin{array}
[c]{c}%
\Delta f-S_{t}\left(  A_{\varepsilon}\left(  x\right)  -A_{\varepsilon}\left(
x\right)  \right)  -G_{t}\left(  p-q\right)  \smallskip\\
k_{f}\Delta f-G_{t}\left(  A_{\varepsilon}\left(  x\right)  -A_{\varepsilon
}\left(  x\right)  \right)  -\Gamma_{t}\left(  p-q\right)
\end{array}
\right)  ,\quad\forall t\in\left[  0,T\right]  ,\text{ }\forall p,q,x,y\in
\mathbb{R}^{2}%
\end{align*}
and we have, for all $t\in\left[  0,T\right]  $,%
\[%
\begin{array}
[c]{ccl}%
\left\vert \Delta B_{\varepsilon}\left(  t\right)  \right\vert  & \leq &
\left(  L_{f}+\dfrac{M_{T}}{\varepsilon}\right)  \left\vert x-y\right\vert
+M_{T}\left\vert p-q\right\vert +L_{f_{phase}}\left(  \left\vert
x-y\right\vert +\left\vert p-q\right\vert \right)  \medskip\\
&  & +\dfrac{M_{T}}{\varepsilon}\left\vert x-y\right\vert +\left\vert
\left\vert \Gamma\right\vert \right\vert _{\infty}\left\vert p-q\right\vert
\medskip\\
& \leq & L_{\varepsilon,T}\left(  \left\vert x-y\right\vert +\left\vert
p-q\right\vert \right)  ,\quad\text{where}\quad L_{\varepsilon,T}%
:=L_{f}+L_{f_{phase}}+\left\vert \left\vert \Gamma\right\vert \right\vert
_{\infty}+M_{T}+\dfrac{2M_{T}}{\varepsilon}.
\end{array}
\]
Analogous, the Lipschitz property of $\Sigma_{\varepsilon}$ easily follows. In
order to have the existence and uniqueness of the solution for each penalized
system (\ref{penalized Phase-Space system}), we need to have the following
condition for the initial data:%
\begin{equation}
\mathbb{E}\left(  \left\vert X_{0}\right\vert ^{2}+\left\vert P_{0}\right\vert
^{2}\right)  <+\infty. \label{boundedness for the initial data}%
\end{equation}
Condition (\ref{boundedness for the initial data}), together with global
Lipschitz properties for the RHS coefficients of
(\ref{penalized Phase-Space system}) (which give linear growth) and the
boundedness of $S,G,\Gamma$ provide a second order boundedness for $\left\vert
B_{\varepsilon}\left(  t,x,p\right)  \right\vert ^{2}$ and $\left\vert
\left\vert \Sigma_{\varepsilon}\left(  t,x,p\right)  \right\vert \right\vert
_{F}^{2}$. However, due to the nature of the reflecting problems we consider
in Section \ref{Geom compatibility}, we need stronger moments for the initial
data:%
\begin{equation}
\mathbb{E}\left(  \left\vert X_{0}\right\vert ^{4}+\left\vert P_{0}\right\vert
^{4}\right)  <+\infty.
\label{boundedness for the initial data due to problems}%
\end{equation}

In conclusion, for every $\varepsilon>0$, there exists a unique solution for
the penalized system (\ref{penalized Phase-Space system}), which is an
$\mathbb{F}$-adapted triplet $\left(  X^{\varepsilon},P^{\varepsilon
},U^{\varepsilon}\right)  ,$%
\begin{equation}
\left\{
\begin{array}
[c]{l}%
\left(  X^{\varepsilon},P^{\varepsilon}\right)  \in S_{2}^{2}\left[
0,T\right]  \times S_{2}^{2}\left[  0,T\right]  ,\quad U^{\varepsilon}\in
L^{2}\left(  \Omega\times\left[  0,T\right]  ;\mathbb{R}^{2}\right)
\medskip\\
U_{t}^{\varepsilon}=\nabla\varphi_{\varepsilon}\left(  X_{t}^{\varepsilon
}\right)  ,\quad a.e.\text{ }t\in\left[  0,T\right]  .
\end{array}
\right.  \label{space solutions for approx. system}%
\end{equation}
Denote, for simplicity,%
\[
\left\{
\begin{array}
[c]{l}%
f_{t}^{\varepsilon}:=f\left(  t,X_{t}^{\varepsilon}\right)  ,\quad
g_{t}^{\varepsilon}:=g\left(  t,X_{t}^{\varepsilon}\right)  ,\medskip\\
f_{ph,t}^{\varepsilon}:=k_{f}~f\left(  t,X_{t}^{\varepsilon},P_{t}%
^{\varepsilon}\right)  \quad\text{and}\quad g_{ph,t}^{\varepsilon}%
:=k_{g}~g\left(  t,X_{t}^{\varepsilon},P_{t}^{\varepsilon}\right)  .
\end{array}
\right.
\]
A possible choice is to consider $f_{phase}\left(  t,x,p\right)
=g_{phase}\left(  t,x,p\right)  \equiv0,$ so the second equation from
(\ref{penalized Phase-Space system}) becomes of homogenous type.

Consider the regular function%
\begin{equation}
\mathcal{H}\left(  x,p\right)  :=\varphi_{\varepsilon}\left(  x\right)
+\frac{1}{2}\left\vert p\right\vert ^{2}. \label{Combined energetic term}%
\end{equation}
In order to compute $d\mathcal{H}\left(  X_{t}^{\varepsilon},P_{t}%
^{\varepsilon}\right)  $, remind that the adjoint $S_{t}^{\ast}=S_{t},$
$G_{t}^{\ast}=-G_{t}$ and $\Gamma_{t}^{\ast}=\Gamma_{t},$ for every
$t\in\left[  0,T\right]  $. Let's write (\ref{penalized Phase-Space system})
into a more suitable form:%
\begin{equation}
\left\{
\begin{array}
[c]{l}%
dX_{t}^{\varepsilon}=\left(  f_{t}^{\varepsilon}-S_{t}U_{t}^{\varepsilon
}-G_{t}P_{t}^{\varepsilon}\right)  dt+g_{t}^{\varepsilon}dW_{t}~,\medskip\\
dP_{t}^{\varepsilon}=\left(  f_{ph,t}^{\varepsilon}-G_{t}U_{t}^{\varepsilon
}-\Gamma_{t}P_{t}^{\varepsilon}\right)  dt+g_{ph,t}^{\varepsilon}dW_{t}~.
\end{array}
\right.  \label{Phase-Space system suitable}%
\end{equation}
We derive the It\^{o}'s formula for (\ref{Combined energetic term}). We have:%
\[%
\begin{array}
[c]{ccl}%
d\mathcal{H}\left(  X_{t}^{\varepsilon},P_{t}^{\varepsilon}\right)  & = &
\left\langle \nabla\varphi_{\varepsilon}\left(  X_{t}^{\varepsilon}\right)
,dX_{t}^{\varepsilon}\right\rangle +\dfrac{1}{2}\mathrm{Tr}\left[
D^{2}\left(  \varphi_{\varepsilon}\left(  X_{t}^{\varepsilon}\right)  \right)
d\left[  X^{\varepsilon}\right]  _{t}\right]  +\left\langle P_{t}%
^{\varepsilon},dP_{t}^{\varepsilon}\right\rangle +\dfrac{1}{2}d\left[
P\right]  _{t}\medskip\\
& = & \left\langle U_{t}^{\varepsilon},f_{t}^{\varepsilon}\right\rangle
dt-\left\langle S_{t}U_{t}^{\varepsilon},U_{t}^{\varepsilon}\right\rangle
dt-\left\langle U_{t}^{\varepsilon},G_{t}P_{t}^{\varepsilon}\right\rangle
dt+\left\langle U_{t}^{\varepsilon},g_{t}^{\varepsilon}dW_{t}\right\rangle
\medskip\\
&  & +\dfrac{1}{2}\mathrm{Tr}\left[  D^{2}\left(  \varphi_{\varepsilon}\left(
X_{t}^{\varepsilon}\right)  \right)  g_{t}^{\varepsilon}\left(  g_{t}%
^{\varepsilon}\right)  ^{\ast}\right]  dt\medskip\\
&  & +\left\langle P_{t}^{\varepsilon},f_{ph,t}^{\varepsilon}\right\rangle
dt-\left\langle P_{t}^{\varepsilon},G_{t}U_{t}^{\varepsilon}\right\rangle
dt-\left\langle \Gamma_{t}P_{t}^{\varepsilon},P_{t}^{\varepsilon}\right\rangle
dt+\left\langle P_{t}^{\varepsilon},g_{ph,t}^{\varepsilon}dW_{t}\right\rangle
\medskip\\
&  & +\dfrac{1}{2}\left\vert \left\vert g_{ph,t}^{\varepsilon}\right\vert
\right\vert ^{2}dt\quad\quad\text{and, rearranging the terms,}\medskip\\
& = & -\left\langle S_{t}U_{t}^{\varepsilon},U_{t}^{\varepsilon}\right\rangle
dt-\left\langle \Gamma_{t}P_{t}^{\varepsilon},P_{t}^{\varepsilon}\right\rangle
dt+\left\langle U_{t}^{\varepsilon},f_{t}^{\varepsilon}\right\rangle
dt+\left\langle P_{t}^{\varepsilon},f_{ph,t}^{\varepsilon}\right\rangle
dt\medskip\\
&  & -\left\langle U_{t}^{\varepsilon},G_{t}P_{t}^{\varepsilon}\right\rangle
dt-\left\langle P_{t}^{\varepsilon},G_{t}U_{t}^{\varepsilon}\right\rangle
dt+\dfrac{1}{2}\mathrm{Tr}\left[  D^{2}\left(  \varphi_{\varepsilon}\left(
X_{t}^{\varepsilon}\right)  \right)  g_{t}^{\varepsilon}\left(  g_{t}%
^{\varepsilon}\right)  ^{\ast}\right]  dt+\dfrac{1}{2}\left\vert \left\vert
g_{ph,t}^{\varepsilon}\right\vert \right\vert ^{2}dt\medskip\\
&  & +\left\langle U_{t}^{\varepsilon},g_{t}^{\varepsilon}dW_{t}\right\rangle
+\left\langle P_{t}^{\varepsilon},g_{ph,t}^{\varepsilon}dW_{t}\right\rangle .
\end{array}
\]
If we focus on the gyroscopic terms,%
\[
-\left\langle U_{t}^{\varepsilon},G_{t}P_{t}^{\varepsilon}\right\rangle
-\left\langle P_{t}^{\varepsilon},G_{t}U_{t}^{\varepsilon}\right\rangle
=-\left\langle U_{t}^{\varepsilon},G_{t}P_{t}^{\varepsilon}\right\rangle
-\left\langle G_{t}^{\ast}P_{t}^{\varepsilon},U_{t}^{\varepsilon}\right\rangle
=-\left\langle U_{t}^{\varepsilon},G_{t}P_{t}^{\varepsilon}\right\rangle
+\left\langle G_{t}P_{t}^{\varepsilon},U_{t}^{\varepsilon}\right\rangle =0.
\]
Consequently, we obtain%
\[%
\begin{array}
[c]{l}%
d\mathcal{H}\left(  X_{t}^{\varepsilon},P_{t}^{\varepsilon}\right)
+\left\langle S_{t}U_{t}^{\varepsilon},U_{t}^{\varepsilon}\right\rangle
dt+\left\langle \Gamma_{t}P_{t}^{\varepsilon},P_{t}^{\varepsilon}\right\rangle
dt\medskip\\
\quad=\left\langle U_{t}^{\varepsilon},f_{t}^{\varepsilon}\right\rangle
dt+k_{f}\left\langle P_{t}^{\varepsilon},f_{t}^{\varepsilon}\right\rangle
dt+\dfrac{1}{2}\mathrm{Tr}\left[  D^{2}\left(  \varphi_{\varepsilon}\left(
X_{t}^{\varepsilon}\right)  \right)  g_{t}^{\varepsilon}\left(  g_{t}%
^{\varepsilon}\right)  ^{\ast}\right]  dt+\dfrac{k_{g}^{2}}{2}\left\vert
\left\vert g_{t}^{\varepsilon}\right\vert \right\vert _{F}^{2}dt\medskip\\
\quad+\left\langle U_{t}^{\varepsilon},g_{t}^{\varepsilon}dW_{t}\right\rangle
+k_{g}\left\langle P_{t}^{\varepsilon},g_{t}^{\varepsilon}dW_{t}\right\rangle
,
\end{array}
\]
or, under integral form,%
\begin{equation}%
\begin{array}
[c]{l}%
\varphi_{\varepsilon}\left(  X_{t}^{\varepsilon}\right)  +\dfrac{1}%
{2}\left\vert P_{t}^{\varepsilon}\right\vert ^{2}+%
{\displaystyle\int_{0}^{t}}
\left\langle S_{s}U_{s}^{\varepsilon},U_{s}^{\varepsilon}\right\rangle ds+%
{\displaystyle\int_{0}^{t}}
\left\langle \Gamma_{s}P_{s}^{\varepsilon},P_{s}^{\varepsilon}\right\rangle
ds\medskip\\
\quad=\varphi_{\varepsilon}\left(  X_{0}^{\varepsilon}\right)  +\dfrac{1}%
{2}\left\vert P_{0}^{\varepsilon}\right\vert ^{2}+%
{\displaystyle\int_{0}^{t}}
\left\langle U_{s}^{\varepsilon},f_{s}^{\varepsilon}\right\rangle ds+%
{\displaystyle\int_{0}^{t}}
k_{f}\left\langle P_{s}^{\varepsilon},f_{s}^{\varepsilon}\right\rangle
ds+\dfrac{k_{g}^{2}}{2}%
{\displaystyle\int_{0}^{t}}
\left\vert \left\vert g_{s}^{\varepsilon}\right\vert \right\vert _{F}%
^{2}ds\medskip\\
\quad+\dfrac{1}{2}%
{\displaystyle\int_{0}^{t}}
\mathrm{Tr}\left[  D^{2}\left(  \varphi_{\varepsilon}\left(  X_{s}%
^{\varepsilon}\right)  \right)  g_{s}^{\varepsilon}\left(  g_{s}^{\varepsilon
}\right)  ^{\ast}\right]  ds\medskip\\
\quad+%
{\displaystyle\int_{0}^{t}}
\left\langle U_{s}^{\varepsilon},g_{s}^{\varepsilon}dW_{s}\right\rangle +k_{g}%
{\displaystyle\int_{0}^{t}}
\left\langle P_{s}^{\varepsilon},g_{s}^{\varepsilon}dW_{s}\right\rangle ,
\end{array}
\label{Eq energy integral}%
\end{equation}

\hfill
\end{proof}

\subsection{Reflection problems with diffusion - frontier geometry
link\label{Geom compatibility}}

\begin{proof}
[Proof (continued)]The term from (\ref{Eq energy integral}) which is difficult
to control is $\mathrm{Tr}\left[  D^{2}\left(  \varphi_{\varepsilon}\left(
X_{t}^{\varepsilon}\right)  \right)  g_{t}^{\varepsilon}\left(  g_{t}%
^{\varepsilon}\right)  ^{\ast}\right]  $. This is why we consider the
situation of a reflection problem, i.e. $\varphi:\mathbb{R}^{2}\rightarrow
(-\infty,+\infty],$ $\varphi=I_{C},$ where $C\subset\mathbb{R}^{2}$ is a
nonempty, closed, convex bounded set and introduce some restrictive dependence
of the diffusion with respect to the geometry of the domain $C$. We impose
that the diffusion is tangent at the boundary and its effect over the
curvature is controlled, not eliminated.

Let us present some general classes of domains which are covered by our study
and introduce some compatibility links between their boundaries and the
diffusion coefficients.$\smallskip$

\underline{\textit{First family of domains.}}\textit{ }We consider domains $C$
featuring smooth enough frontier, more precisely, of class $C^{2}$. For
$x\in\partial C,$ denote by $n\left(  x\right)  $ or $n_{x}$ the exterior
normal in $x$ and define on the tangent space $T_{x}\left(  \partial C\right)
$ the bilinear symmetric second fundamental form $\mathbf{II}_{x}:T_{x}\left(
\partial C\right)  \times T_{x}\left(  \partial C\right)  \rightarrow
\mathbb{R},$%
\[
\mathbf{II}_{x}\left(  v,w\right)  =\left\langle \nabla_{v}n\left(  x\right)
,w\right\rangle ,
\]
where $\nabla_{v}n$ is the directional derivative of the normal field $n$ on
the direction of the tangent vector $v$. For an orthonormal base
$\{e_{i}\}_{i}$ of the tangent space, the matrix of $\mathbf{II}_{x}$ is
$\left(  \left\langle \nabla_{e_{i}}v,e_{j}\right\rangle \right)  _{i,j}$ and
$\mathbf{II}_{x}$ is a measure of the intensity of the curvature of the
boundary along the tangent direction $v$. Using the scalar product on the
tangent set, we associate to the bilinear form $\mathbf{II}_{x}$ the linear
self-adjoint shape operator $\mathbf{S}_{x}:T_{x}\left(  \partial C\right)
\rightarrow T_{x}\left(  \partial C\right)  ,$ defined by%
\[
\mathbf{II}_{x}\left(  v,w\right)  =\left\langle D_{v}n\left(  x\right)
,w\right\rangle =\left\langle \nabla_{v}n\left(  x\right)  ,w\right\rangle
=\left\langle \mathbf{S}_{x}v,w\right\rangle .
\]

For our problem, for each $x\in\partial C,$ let us define the diffusion
covariance by%
\[
a\left(  t,x,p\right)  :=g\left(  t,x,p\right)  g\left(  t,x,p\right)  ^{\ast
}.
\]
We impose the following assumptions, which hold uniformly with respect to the
rest of the involved variables,%
\begin{equation}%
\begin{array}
[c]{l}%
\left(  H_{\left(  G,\partial C\right)  }\right)  :\medskip\\
\left\{
\begin{array}
[c]{ll}%
\left(  i\right)  & g\left(  t,x,p\right)  ^{\ast}n\left(  x\right)
=0,\quad\forall x\in\partial C\quad\text{(tangency condition)}\medskip\\
\left(  ii\right)  & \left\vert \mathrm{Tr}_{T_{x}\left(  \Gamma_{r}\right)
}\left[  \mathbf{S}_{x}^{\left(  r\right)  }P_{T_{x}\Gamma_{r}}a\left(
t,x,p\right)  P_{T_{x}\Gamma_{r}}\right]  \right\vert \leq\mathcal{C}%
_{r}\mathcal{~}\mathrm{Tr}\left[  a\left(  t,x,p\right)  \right]
\quad\text{(curvature condition),}%
\end{array}
\right.
\end{array}
\label{Class C2 geom cond}%
\end{equation}
where $P_{T}\left(  y\right)  =I-n\left(  y\right)  \otimes n\left(  y\right)
$ is the projection. The first condition from (\ref{Class C2 geom cond})
states that the noise does not has a normal component on the frontier. The
second intends to control the curvature effect by the total intensity of the
diffusion. By $\mathbf{S}_{x}^{\left(  r\right)  }$ we understand the shape
operator of the $r-$parallel surface to $\partial C$, containing $x$. More
precisely,%
\[
\mathbf{S}_{x}^{\left(  r\right)  }=\mathbf{S}_{p}\left(  I+r\mathbf{S}%
_{p}\right)  ^{-1},\quad\text{with }p=\Pi_{C}\left(  x\right)  \in\partial C
\]
is the shape operator of $\Gamma_{r}$ in $x$, where, for $r>0$, $\Gamma
_{r}:=\{x=y+rn\left(  y\right)  :y\in\partial C\}$. More details on this
structural characterization will be discussed below. This is not a viability
condition of the domain since we do not impose any additional assumption on
the drift, the product $\left\langle n\left(  x\right)  ,f\left(
t,x,p\right)  \right\rangle $ accepting positive values, with the effect of
pushing the process outside the domain. This effect will be compensated by the
intervention of the reflection.$\smallskip$

\textbf{The impact on the term} $\mathrm{Tr}\left[  D^{2}\left(
\varphi_{\varepsilon}\left(  X_{t}^{\varepsilon}\right)  \right)
g_{t}^{\varepsilon}\left(  g_{t}^{\varepsilon}\right)  ^{\ast}\right]  .$

The regularity of the frontier permits us to chose an exterior tubular
$r_{0}-$neighborhood of the domain%
\[
\mathcal{U}_{r_{0}}^{ext}\left(  \partial C\right)  =\{x\notin C:dist\left(
x,\partial C\right)  \}<r_{0}.
\]
The constant $r_{0}$ is chosen such that, for every $x\in\mathcal{U}_{r_{0}%
}^{ext}\left(  \partial C\right)  ,$ there exists a unique projection
$p=p_{x}:=\Pi_{C}\left(  x\right)  \in\partial C.$ Denote also $r=r_{x}%
:=dist\left(  x,C\right)  .$ Therefore, $x=p+r~n\left(  p\right)  .$
Concerning our functions and operators, we have:%
\[
q_{C}\left(  x\right)  :=\frac{1}{2}dist^{2}\left(  x,C\right)  =\frac{1}%
{2}\left\vert x-\Pi_{C}\left(  x\right)  \right\vert ^{2},\quad\varphi
_{\varepsilon}\left(  x\right)  =\frac{1}{\varepsilon}q_{C}\left(  x\right)
=\frac{1}{2\varepsilon}dist^{2}\left(  x,C\right)
\]
and, in the exterior of $C$,%
\begin{equation}
\left\{
\begin{array}
[c]{ccl}%
\nabla q_{C}\left(  x\right)  & = & x-\Pi_{C}\left(  x\right)  =x-p=r~n\left(
p\right)  \medskip\\
\nabla\varphi_{\varepsilon}\left(  x\right)  & = & \dfrac{r}{\varepsilon
}n\left(  p\right)  =\dfrac{1}{\varepsilon}\left(  x-\Pi_{C}\left(  x\right)
\right)  =U_{\varepsilon}\left(  x\right)  .
\end{array}
\right.  \label{gradients formulas}%
\end{equation}
We first observe that $D^{2}q_{C}\left(  x\right)  =I-D\Pi_{C}\left(
x\right)  .$ Using tubular coordinates, if $x=F\left(  p,r\right)
=p+rn\left(  p\right)  ,$ then the derivative along the tangent directions is
$D_{p}F=I+r~Dn\left(  p\right)  .$ Our shape operator associated to the second
fundamental form is now $\mathbf{S}_{p}\left(  v\right)  =D_{v}n\left(
p\right)  ,$ for $v\in T_{p}\left(  \partial C\right)  .$ Consequently,%
\[
D\Pi_{C}\left(  x\right)  =\left(  I+r\mathbf{S}_{p}\right)  ^{-1}P_{T}\left(
p\right)
\]
and%
\[%
\begin{array}
[c]{ccl}%
D^{2}q_{C}\left(  x\right)  & = & I-\left(  I+r\mathbf{S}_{p}\right)
^{-1}P_{T}\left(  p\right)  \medskip\\
& = & P_{T}\left(  p\right)  +n\left(  p\right)  \otimes n\left(  p\right)
-\left(  I+r\mathbf{S}_{p}\right)  ^{-1}P_{T}\left(  p\right)  \medskip\\
& = & n\left(  p\right)  \otimes n\left(  p\right)  +\left[  I-\left(
I+r\mathbf{S}_{p}\right)  ^{-1}\right]  P_{T}\left(  p\right)  \medskip\\
& = & n\left(  p\right)  \otimes n\left(  p\right)  +r\mathbf{S}_{p}\left(
I+r\mathbf{S}_{p}\right)  ^{-1}P_{T}\left(  p\right)  .
\end{array}
\]
Hence, since $\varphi_{\varepsilon}=\frac{1}{\varepsilon}q_{C},$ we obtain,
for $x\in\mathcal{U}_{r_{0}}^{ext}\left(  \partial C\right)  ,$
$0<r=dist\left(  x,C\right)  <r_{0},$%
\begin{equation}
D^{2}\varphi_{\varepsilon}\left(  x\right)  =\dfrac{1}{\varepsilon}\left[
n\left(  p\right)  \otimes n\left(  p\right)  +r\mathbf{S}_{p}\left(
I+r\mathbf{S}_{p}\right)  ^{-1}P_{T}\left(  p\right)  \right]  .
\label{the form of D fi_eps}%
\end{equation}
For the normal component, $\tfrac{1}{\varepsilon}n\left(  p\right)  \otimes
n\left(  p\right)  $ from (\ref{the form of D fi_eps}), for every matrix
$a=g\left(  t,x,z\right)  g\left(  t,x,z\right)  ^{\ast},$%
\begin{equation}
\dfrac{1}{\varepsilon}\mathrm{Tr}\left[  \left(  n\left(  p\right)  \otimes
n\left(  p\right)  \right)  a\right]  =\dfrac{1}{\varepsilon}n\left(
p\right)  ^{\ast}g\left(  t,x,z\right)  g\left(  t,x,z\right)  ^{\ast}n\left(
p\right)  =\dfrac{1}{\varepsilon}\left\vert g\left(  t,x,z\right)  ^{\ast
}n\left(  p\right)  \right\vert ^{2}. \label{for the normal comp}%
\end{equation}
But, according to (\ref{Class C2 geom cond}-$\left(  i\right)  $), $g^{\ast
}n=0$ on $\partial C;$ its Lipschitz property in $x,$ uniformly with respect
to $t$ and $z$, leads to%
\begin{align*}
g\left(  t,x,z\right)  ^{\ast}n\left(  p\right)   &  =\left(  g\left(
t,x,z\right)  -g\left(  t,p,z\right)  \right)  ^{\ast}n\left(  p\right)
+g\left(  t,p,z\right)  ^{\ast}n\left(  p\right) \\
&  =\left(  g\left(  t,x,z\right)  -g\left(  t,p,z\right)  \right)  ^{\ast
}n\left(  p\right)  .
\end{align*}
Therefore, since $\left\vert n\left(  p\right)  \right\vert =1$,%
\[
\left\vert g\left(  t,x,z\right)  ^{\ast}n\left(  p\right)  \right\vert
\leq\left\vert \left\vert \left(  g\left(  t,x,z\right)  -g\left(
t,p,z\right)  \right)  ^{\ast}\right\vert \right\vert _{\mathrm{op}}\leq
L_{g}\left\vert x-p\right\vert =L_{g}r
\]
Geometrically, on $\partial C$, the columns of $g$ are tangent to $\partial
C$. When we drive to $x=p+r~n\left(  p\right)  ,$ the deviation of $g\left(
t,x,z\right)  $ from the tangent in $p$ is of $r-$scale. Therefore,
(\ref{for the normal comp}) yields, for $0<\varepsilon\leq1$,%
\begin{equation}
\dfrac{1}{\varepsilon}\mathrm{Tr}\left[  \left(  n\left(  p\right)  \otimes
n\left(  p\right)  \right)  gg^{\ast}\right]  \leq r\dfrac{r}{\varepsilon
}L_{g}^{2}\overset{(\ref{gradients formulas})}{=}L_{g}^{2}\varepsilon
\left\vert U_{\varepsilon}\left(  x\right)  \right\vert ^{2}\leq L_{g}%
^{2}\left\vert U_{\varepsilon}\left(  x\right)  \right\vert ^{2}
\label{for the normal comp2}%
\end{equation}

For the curved component from (\ref{the form of D fi_eps}), suppose that the
curvature $\mathbf{S}_{p}\left(  I+r\mathbf{S}_{p}\right)  ^{-1}$ is bounded
into a small enough tubular neighborhood of the domain by a positive constant
$C_{\partial C}$:%
\begin{equation}
C_{\partial C}:=\sup_{p\in\partial C,0\leq r<r_{0}}\left\vert \left\vert
\mathbf{S}_{p}\left(  I+r\mathbf{S}_{p}\right)  ^{-1}\right\vert \right\vert .
\label{boundedness of curvature}%
\end{equation}
Also, on the same neighborhood, $A:=P_{T}\left(  p\right)  a\left(
t,x,p\right)  P_{T}\left(  p\right)  \geq0$. Therefore,%
\begin{equation}%
\begin{array}
[c]{ccl}%
\left\vert \mathrm{Tr}\left[  \mathbf{S}_{p}\left(  I+r\mathbf{S}_{p}\right)
^{-1}P_{T}\left(  p\right)  a\left(  t,x,p\right)  P_{T}\left(  p\right)
\right]  \right\vert  & \leq & C_{\partial C}\mathrm{Tr}\left[  P_{T}%
aP_{T}\right]  \leq C_{\partial C}\mathrm{Tr}\left[  P_{T}a\right]  \medskip\\
& \leq & C_{\partial C}\mathrm{Tr}\left[  a\right]  =C_{\partial C}\left\vert
\left\vert g\left(  t,x,p\right)  \right\vert \right\vert _{F}^{2}~.
\end{array}
\label{bound for curvature part}%
\end{equation}
In fact, (\ref{Class C2 geom cond}-$\left(  ii\right)  $) can be avoid if we
have (\ref{boundedness of curvature}). Otherwise, (\ref{Class C2 geom cond}%
-$\left(  ii\right)  $) should be invoked, together with the link between $x$
and its projection $p$.

Summarizing, from (\ref{the form of D fi_eps}), (\ref{for the normal comp2})
and (\ref{bound for curvature part}) we obtain\footnote{The operator
$\mathbf{S}_{p}$ acts only on the tangent space. Also, $\mathbf{S}_{p}\left(
I+r\mathbf{S}_{p}\right)  ^{-1}P_{T}$ is a tangent operator. Therefore, we can
also introduce $P_{T}$ inside the trace, on the right side, without changes.}%
\begin{equation}%
\begin{array}
[c]{l}%
\mathrm{Tr}\left[  D^{2}\varphi_{\varepsilon}\left(  x\right)  a\left(
t,x,p\right)  \right]  \medskip\\
\quad\quad\quad=\dfrac{1}{\varepsilon}\mathrm{Tr}\left[  \left(  n\left(
p\right)  \otimes n\left(  p\right)  \right)  a\left(  t,x,p\right)  \right]
+\frac{r}{\varepsilon}\mathrm{Tr}\left[  \mathbf{S}_{p}\left(  I+r\mathbf{S}%
_{p}\right)  ^{-1}P_{T}\left(  p\right)  a\left(  t,x,p\right)  \right]
\medskip\\
\quad\quad\quad=\dfrac{1}{\varepsilon}\left\vert g\left(  t,x,p\right)
^{\ast}n\left(  p\right)  \right\vert ^{2}+\dfrac{r}{\varepsilon}%
\mathrm{Tr}\left[  \mathbf{S}_{p}\left(  I+r\mathbf{S}_{p}\right)  ^{-1}%
P_{T}\left(  p\right)  a\left(  t,x,p\right)  P_{T}\left(  p\right)  \right]
\medskip\\
\quad\quad\quad\leq\varepsilon L_{g}^{2}\left\vert U_{\varepsilon}\left(
x\right)  \right\vert ^{2}+C_{\partial C}\left\vert U_{\varepsilon}\left(
x\right)  \right\vert \left\vert \left\vert g\left(  t,x,p\right)  \right\vert
\right\vert _{F}^{2}\medskip\\
\quad\quad\quad\leq\left(  \varepsilon L_{g}^{2}+\delta\right)  \left\vert
U_{\varepsilon}\left(  x\right)  \right\vert ^{2}+\dfrac{C_{\partial C}^{2}%
}{4\delta}\left\vert \left\vert g\left(  t,x,p\right)  \right\vert \right\vert
_{F}^{4}~,\quad\text{for }\delta>0.
\end{array}
\label{estimates of the diff term_C2}%
\end{equation}
So, if consider fourth-order moments finite for the initial data, a linear
growth for $g$, fourth order moments for $g$ are bounded. Therefore, one can
choose $\delta>0$ such that an absorption argument involving $\left\vert
U_{\varepsilon}\left(  x\right)  \right\vert ^{2}$ from the RHS and the
coercivity of $S$ from the LHS of (\ref{Eq energy integral}) can be used.

\underline{\textit{Second family of domains.}}\textit{ }Polyhedral sets. We
mention that the additional assumptions (\ref{Class C2 geom cond}) are not
generally verified if $C$ is a polyhedral set since they are not well defined
at edges and vertexes, where $\partial C$ is not of class $C^{2}$. This is why
we construct some adapted one, which is more suitable for these type of convex
domains. Consider the polyhedral set%
\[
C:=%
{\displaystyle\bigcap_{i=1}^{N}}
\left\{  x\in\mathbb{R}^{2}:l_{i}\left(  x\right)  \leq0\right\}  ,
\]
where, for all $i$, $l_{i}\left(  x\right)  =\left\langle n_{i},x\right\rangle
-c_{i}$ and $n_{i}$ is the unitary exterior normal at the face $F_{i}=\{x\in
C:l_{i}\left(  x\right)  =0\}.$ Denote, for ever $x\in C,$%
\[
I\left(  x\right)  :=\{i=\overline{1,N}:l_{i}\left(  x\right)  =0\}
\]
the set of active faces for $x$. The exterior normal cone at $C$ in $x$ is%
\[
N_{C}\left(  x\right)  =\mathrm{cone}\{n_{i}:i\in I\left(  x\right)
\}=\left\{  \sum_{i\in I\left(  x\right)  }\lambda_{i}n_{i}:\lambda_{i}%
\geq0\right\}  .
\]

On every face $F_{i}$, the tangency condition (\ref{Class C2 geom cond}%
-$\left(  i\right)  $) becomes%
\begin{equation}
g\left(  t,x,p\right)  ^{\ast}n_{i}=0,\quad\forall x\in ris\left(
F_{i}\right)  . \label{Poly geom tang cond}%
\end{equation}
Geometrically, it means that the noise does not propagate transversely across
the face $F_{i}~.$ On the relative interior of faces,
(\ref{Class C2 geom cond}-$\left(  ii\right)  $) is trivial since the
curvature is $0$ on planes and, therefore, $\mathbf{II}_{F_{i}}=0$ on
$ris\left(  F_{i}\right)  ,$ for every $i=\overline{1,N}$. On edges and
vertexes, (\ref{Poly geom tang cond}) must take place on the set of active
faces. Geometrically, it is equivalent with the fact that the noise is tangent
at all active faces in $x$:%
\[
\mathrm{Ran}g\left(  t,x,p\right)  \subseteq%
{\displaystyle\bigcap_{i\in I\left(  x\right)  }}
n_{i}^{\bot}~.
\]
Remark that, at vertexes, this condition can prove to be very restrictive. For
example, if $C=\mathbb{R}_{+}^{2}$~, $g\left(  t,0,p\right)  ^{\ast}e_{i}=0$
for every $i,$ which gives $g\left(  t,0,p\right)  =0$. However, this is not a
complete viability condition since it does not restrict the drift.

The reflection at the frontier is made by the directions $d_{i}\in
\mathbb{R}^{2}$, one for every face and is produced by the bounded variation
increasing process $K^{i},$ $dK^{i}$ taking effect when $X_{t}\in F_{i}.$ On
edges and vertexes, the combined reflection is%
\[
dK_{t}=\sum_{i\in I\left(  x\right)  }d_{i}dK_{t}^{i}\in\mathrm{cone}%
\{n_{i}:i\in I\left(  X_{t}\right)  \}.
\]
For normal reflection on faces we ask $d_{i}=-n_{i}.$ This is not enough on
edges and vertexes, the system of active directions must generate a well posed
reflection. For an active set of indices $I$, consider the reflection matrix
$R_{I}:=\left[  d_{i}\right]  _{i\in I}$ and ask the following condition: for
every $J\subset I,$ there exists a vector $u_{j}>0$ such that
\begin{equation}
R_{J}u_{J}>0,\quad\text{component-wise.} \label{condition completely-S}%
\end{equation}
This means that the reflecting directions can be positively combined in order
to produce a vector which is strictly interior to faces
intersection.$\smallskip$

\textbf{The impact on the term} $\mathrm{Tr}\left[  D^{2}\left(
\varphi_{\varepsilon}\left(  X_{t}^{\varepsilon}\right)  \right)
g_{t}^{\varepsilon}\left(  g_{t}^{\varepsilon}\right)  ^{\ast}\right]  .$ In
this case, a penalization on faces is recommended:%
\begin{equation}
\left\{
\begin{array}
[c]{l}%
\Psi_{\varepsilon}\left(  x\right)  =\dfrac{1}{2\varepsilon}%
{\displaystyle\sum\limits_{i=1}^{N}}
\left(  l_{i}\left(  x\right)  ^{+}\right)  ^{2}\quad\text{and}\quad
U_{\varepsilon}\left(  x\right)  =\nabla\Psi_{\varepsilon}\left(  x\right)
=\dfrac{1}{\varepsilon}%
{\displaystyle\sum\limits_{i=1}^{N}}
l_{i}\left(  x\right)  ^{+}n_{i}\medskip\\
\text{For oblique reflection on faces we use }U_{\varepsilon}^{R}\left(
x\right)  =\nabla\Psi_{\varepsilon}\left(  x\right)  =\dfrac{1}{\varepsilon}%
{\displaystyle\sum\limits_{i=1}^{N}}
l_{i}\left(  x\right)  ^{+}d_{i}\medskip\\
\text{On faces:}\quad D^{2}\Psi_{\varepsilon}\left(  x\right)  =\dfrac
{1}{\varepsilon}%
{\displaystyle\sum\limits_{\{i:l_{i}\left(  x\right)  >0\}}}
n_{i}\otimes n_{i}\medskip\\
\mathrm{Tr}\left[  D^{2}\Psi_{\varepsilon}\left(  x\right)  \left(  x\right)
gg^{\ast}\right]  =\dfrac{1}{\varepsilon}%
{\displaystyle\sum\limits_{\{i:l_{i}\left(  x\right)  >0\}}}
\left\vert g\left(  t,x,p\right)  ^{\ast}n_{i}\right\vert ^{2}.
\end{array}
\right.  \label{penalization for polyhedral sets}%
\end{equation}
But, since, on each face $g\left(  t,x,p\right)  ^{\ast}n_{i}=0$ (when
$l_{i}\left(  x\right)  =0$), and $g$ is Lipschitz it follows, for $x$ close
to $F_{i}$,%
\begin{equation}
\mathrm{Tr}\left[  D^{2}\Psi_{\varepsilon}\left(  x\right)  \left(  x\right)
gg^{\ast}\right]  \leq\dfrac{1}{\varepsilon}%
{\displaystyle\sum\limits_{i}}
L_{i}^{2}\left(  l_{i}\left(  x\right)  ^{+}\right)  ^{2}\leq C_{pol}%
\varepsilon\left\vert U_{\varepsilon}\left(  x\right)  \right\vert ^{2}\leq
C_{pol}\left\vert U_{\varepsilon}\left(  x\right)  \right\vert ^{2}.
\label{estimates of the diff term_poly}%
\end{equation}
Summarizing, at vertexes, we need%
\[
\left\{
\begin{array}
[c]{ll}%
\text{for the diffusion} & g\left(  t,x,p\right)  ^{\ast}n_{i}=0,\quad\forall
i\in I\left(  x\right)  \medskip\\
\text{for the reflection} & d_{i}\cdot\left(  -n_{i}\right)  >0,\quad\forall
i\in I\left(  x\right)  \medskip\\
\text{the compatibility of the family }\left(  d_{i}\right)  _{i\in I\left(
x\right)  }: & \text{(\ref{condition completely-S}).}%
\end{array}
\right.
\]

\hfill
\end{proof}

\subsection{Apriori estimates for the approximating phase-space system}

\begin{proof}
[Proof (continued)]Working in the framework developed in Section
\ref{Geom compatibility}, we return now to the energy balance
(\ref{Eq energy integral}). The compatibilities we developed between the
diffusion and the geometry of the domain provide some estimates for the Trace
term from the energy balance: (\ref{estimates of the diff term_C2}) and
(\ref{estimates of the diff term_poly}). First, we remind that $S$ is $s_{0}%
-$coercive and $\Gamma$ is $\gamma_{0}-$coercive. Consequently if we consider
obstacle problems with $C^{2}$-class boundary, (\ref{Eq energy integral})
yields, for some Young constants $\eta_{1}$ and $\eta_{2}$,%
\[%
\begin{array}
[c]{l}%
\varphi_{\varepsilon}\left(  X_{t}^{\varepsilon}\right)  +\dfrac{1}%
{2}\left\vert P_{t}^{\varepsilon}\right\vert ^{2}+\left(  s_{0}-\eta
_{1}-\dfrac{\varepsilon L_{g}^{2}+\delta}{2}\right)
{\displaystyle\int_{0}^{t}}
\left\vert U_{s}^{\varepsilon}\right\vert ^{2}ds+\left(  \gamma_{0}-\eta
_{2}\right)
{\displaystyle\int_{0}^{t}}
\left\vert P_{s}^{\varepsilon}\right\vert ^{2}ds\medskip\\
\quad\quad\quad\quad\leq\varphi_{\varepsilon}\left(  X_{0}\right)  +\dfrac
{1}{2}\left\vert P_{0}\right\vert ^{2}+\left(  \dfrac{1}{4\eta_{1}}%
+\dfrac{k_{f}^{2}}{4\eta_{2}}\right)
{\displaystyle\int_{0}^{t}}
\left\vert f_{s}^{\varepsilon}\right\vert ^{2}ds+\dfrac{k_{g}^{2}}{2}%
{\displaystyle\int_{0}^{t}}
\left\vert \left\vert g_{s}^{\varepsilon}\right\vert \right\vert _{F}%
^{2}ds\medskip\\
\quad\quad\quad\quad+\dfrac{C_{\partial C}^{2}}{8\delta}%
{\displaystyle\int_{0}^{t}}
\left\vert \left\vert g_{s}^{\varepsilon}\right\vert \right\vert _{F}^{4}ds+%
{\displaystyle\int_{0}^{t}}
\left\langle U_{s}^{\varepsilon},g_{s}^{\varepsilon}dW_{s}\right\rangle +k_{g}%
{\displaystyle\int_{0}^{t}}
\left\langle P_{s}^{\varepsilon},g_{s}^{\varepsilon}dW_{s}\right\rangle .
\end{array}
\]
Denote, for simplicity,%
\begin{equation}
\left\{
\begin{array}
[c]{l}%
c_{U}^{\varepsilon}:=s_{0}-\eta_{1}-\dfrac{\varepsilon L_{g}^{2}+\delta}%
{2},\quad C_{P}:=\gamma_{0}-\eta_{2},\quad C_{f}:=\dfrac{1}{4\eta_{1}}%
+\dfrac{k_{f}^{2}}{4\eta_{2}},\quad C_{g,4}:=\dfrac{C_{\partial C}^{2}%
}{8\delta},\\
M_{t}^{\varepsilon}:=%
{\displaystyle\int_{0}^{t}}
\left\langle U_{s}^{\varepsilon},g_{s}^{\varepsilon}dW_{s}\right\rangle +k_{g}%
{\displaystyle\int_{0}^{t}}
\left\langle P_{s}^{\varepsilon},g_{s}^{\varepsilon}dW_{s}\right\rangle .
\end{array}
\right.  \label{Coeff notations}%
\end{equation}
The above inequality becomes%
\begin{equation}%
\begin{array}
[c]{l}%
\varphi_{\varepsilon}\left(  X_{t}^{\varepsilon}\right)  +\dfrac{1}%
{2}\left\vert P_{t}^{\varepsilon}\right\vert ^{2}+c_{U}^{\varepsilon}%
{\displaystyle\int_{0}^{t}}
\left\vert U_{s}^{\varepsilon}\right\vert ^{2}ds+C_{P}%
{\displaystyle\int_{0}^{t}}
\left\vert P_{s}^{\varepsilon}\right\vert ^{2}ds\medskip\\
\quad\quad\leq\varphi_{\varepsilon}\left(  X_{0}\right)  +\dfrac{1}%
{2}\left\vert P_{0}\right\vert ^{2}+C_{f}%
{\displaystyle\int_{0}^{t}}
\left\vert f_{s}^{\varepsilon}\right\vert ^{2}ds+\dfrac{k_{g}^{2}}{2}%
{\displaystyle\int_{0}^{t}}
\left\vert \left\vert g_{s}^{\varepsilon}\right\vert \right\vert _{F}%
^{2}ds+C_{g,4}%
{\displaystyle\int_{0}^{t}}
\left\vert \left\vert g_{s}^{\varepsilon}\right\vert \right\vert _{F}%
^{4}ds+M_{t}^{\varepsilon}.
\end{array}
\label{Eq for apriori estimates}%
\end{equation}
In order to employ some absorption arguments in
(\ref{Eq for apriori estimates}) we have to calibrate the positive constants
appearing in (\ref{Coeff notations}). Observe that, for $0<\varepsilon
\leq\varepsilon_{0}$, $c_{U}^{\varepsilon}\geq c_{U}^{\varepsilon_{0}}=:C_{U}%
$. It is mandatary that $C_{U}>0$ and $C_{P}>0$. We can replace, for an
uniformly character with respect to the boundedness constant, $c_{U}%
^{\varepsilon}$ by $C_{U}$ in (\ref{Eq for apriori estimates}). Also, for
$X_{0}\in C$, $\varphi_{\varepsilon}\left(  X_{0}\right)  =0$.

In the same manner, if we work with polyhedral sets, we obtain similar
estimates, but the absorption condition becomes now:%
\[
s_{0}-\eta_{1}-\frac{C_{pol}}{2}>0.
\]
The boundedness $\left(  H_{f,g}\right)  $ of $f$ and $g$ on $C$ implies%
\[
\left\{
\begin{array}
[c]{l}%
{\displaystyle\int_{0}^{t}}
\left\vert f_{s}^{\varepsilon}\right\vert ^{2}ds\leq%
{\displaystyle\int_{0}^{t}}
\left(  f^{\#}\left(  s\right)  \right)  ^{2}ds,\quad%
{\displaystyle\int_{0}^{t}}
\left\vert \left\vert g_{s}^{\varepsilon}\right\vert \right\vert _{F}%
^{4}ds\leq%
{\displaystyle\int_{0}^{t}}
\left(  g^{\#}\left(  s\right)  \right)  ^{4}ds,\medskip\\%
{\displaystyle\int_{0}^{t}}
\left\vert \left\vert g_{s}^{\varepsilon}\right\vert \right\vert _{F}%
^{2}ds\leq T^{1/2}\left(
{\displaystyle\int_{0}^{T}}
\left(  g^{\#}\left(  s\right)  \right)  ^{4}ds\right)  ^{1/2}.
\end{array}
\right.
\]
Therefore, there exists an $\varepsilon-$independent $C_{T}=C_{T}\left(
C_{f},k_{g}^{2},C_{g,4}\right)  >0$ such that%
\[
\varphi_{\varepsilon}\left(  X_{t}^{\varepsilon}\right)  +\dfrac{1}%
{2}\left\vert P_{t}^{\varepsilon}\right\vert ^{2}+C_{U}%
{\displaystyle\int_{0}^{t}}
\left\vert U_{s}^{\varepsilon}\right\vert ^{2}ds+C_{P}%
{\displaystyle\int_{0}^{t}}
\left\vert P_{s}^{\varepsilon}\right\vert ^{2}ds\leq\dfrac{1}{2}\left\vert
P_{0}\right\vert ^{2}+C_{T}+M_{t}^{\varepsilon}.
\]
The Burkholder--Davis--Gundy inequality applied to $M_{r}^{\varepsilon},$
gives us, for every $t$,%
\begin{equation}%
\begin{array}
[c]{ccl}%
\mathbb{E}\left(  \sup\limits_{r\in\left[  0,t\right]  }\left\vert
M_{r}^{\varepsilon}\right\vert \right)  & \leq & \sqrt{3}\mathbb{E}\left[
\left(
{\displaystyle\int_{0}^{t}}
\left\vert \left(  g_{s}^{\varepsilon}\right)  ^{\ast}U_{s}^{\varepsilon
}\right\vert ^{2}ds\right)  ^{1/2}\right]  +\sqrt{3}\left\vert k_{g}%
\right\vert \mathbb{E}\left[  \left(
{\displaystyle\int_{0}^{t}}
\left\vert \left(  g_{s}^{\varepsilon}\right)  ^{\ast}P_{s}^{\varepsilon
}\right\vert ^{2}ds\right)  ^{1/2}\right]  \medskip\\
& \leq & \sqrt{3}\eta_{3}\mathbb{E}\left(
{\displaystyle\int_{0}^{t}}
\left\vert U_{s}^{\varepsilon}\right\vert ^{2}ds\right)  \mathbb{+}%
\dfrac{\sqrt{3}}{4\eta_{3}}\mathbb{E}\left(
{\displaystyle\int_{0}^{T}}
\left(  g^{\#}\left(  s\right)  \right)  ^{2}ds\right)  \medskip\\
& + & \sqrt{3}\left\vert k_{g}\right\vert \eta_{4}\mathbb{E}\left(
\sup\limits_{r\in\left[  0,t\right]  }\left\vert P_{r}^{\varepsilon
}\right\vert ^{2}\right)  +\dfrac{\sqrt{3}\left\vert k_{g}\right\vert }%
{4\eta_{4}}\mathbb{E}\left(
{\displaystyle\int_{0}^{t}}
\left\vert \left\vert g_{s}^{\varepsilon}\right\vert \right\vert _{F}%
^{2}ds\right)  \medskip\\
& \leq & \eta_{U}\mathbb{E}\left(
{\displaystyle\int_{0}^{t}}
\left\vert U_{s}^{\varepsilon}\right\vert ^{2}ds\right)  +\eta_{P}%
\mathbb{E}\left(  \sup\limits_{r\in\left[  0,t\right]  }\left\vert
P_{r}^{\varepsilon}\right\vert ^{2}\right)  +C_{T,f,g},
\end{array}
\label{BDG estimate}%
\end{equation}
where $\eta_{U},\eta_{P}>0$ can be controlled to suitable small values.

We write (\ref{Eq for apriori estimates}) on the time interval $\left[
0,r\right]  $ and use the positivity of the terms from the LHS to isolate the
only quantity we want to control. Take $\sup_{0\leq r\leq t}$ and apply
$\mathbb{E}$ for each quantity we analyze. We obtain:%
\[%
\begin{array}
[c]{l}%
\dfrac{1}{2}\mathbb{E}\left(  \sup\limits_{0\leq r\leq t}\left\vert
P_{r}^{\varepsilon}\right\vert ^{2}\right)  \medskip\\
\quad\quad\leq\dfrac{\mathbb{E}\left\vert P_{0}\right\vert ^{2}}{2}%
+C_{f}\mathbb{E}\left(
{\displaystyle\int_{0}^{t}}
\left\vert f_{s}^{\varepsilon}\right\vert ^{2}ds\right)  +\dfrac{k_{g}^{2}}%
{2}\mathbb{E}\left(
{\displaystyle\int_{0}^{t}}
\left\vert \left\vert g_{s}^{\varepsilon}\right\vert \right\vert _{F}%
^{2}ds\right)  +C_{g,4}\mathbb{E}\left(
{\displaystyle\int_{0}^{t}}
\left\vert \left\vert g_{s}^{\varepsilon}\right\vert \right\vert _{F}%
^{4}ds\right)  +\mathbb{E}\left(  \sup\limits_{r\in\left[  0,t\right]
}\left\vert M_{r}^{\varepsilon}\right\vert \right)  \medskip\\
\quad\quad\leq\dfrac{\mathbb{E}\left\vert P_{0}\right\vert ^{2}}{2}+\eta
_{U}\mathbb{E}\left(
{\displaystyle\int_{0}^{t}}
\left\vert U_{s}^{\varepsilon}\right\vert ^{2}ds\right)  +\eta_{P}%
\mathbb{E}\left(  \sup\limits_{r\in\left[  0,t\right]  }\left\vert
P_{r}^{\varepsilon}\right\vert ^{2}\right)  +C_{T,f,g}%
\end{array}
\]
Reasoning in exactly the same manner for the term with $\varphi_{\varepsilon}%
$, we obtain, by adding the estimates,%
\begin{equation}
\mathbb{E}\left(  \sup\limits_{0\leq r\leq t}\varphi_{\varepsilon}\left(
X_{r}^{\varepsilon}\right)  \right)  +\left(  \dfrac{1}{2}-2\eta_{P}\right)
\mathbb{E}\left(  \sup\limits_{0\leq r\leq t}\left\vert P_{r}^{\varepsilon
}\right\vert ^{2}\right)  \leq\mathbb{E}\left\vert P_{0}\right\vert ^{2}%
+2\eta_{U}\mathbb{E}\left(
{\displaystyle\int_{0}^{t}}
\left\vert U_{s}^{\varepsilon}\right\vert ^{2}ds\right)  +2C_{T,f,g}.
\label{prebound of Efieps,EPeps}%
\end{equation}
The choice of $\eta_{P}<1/4$ is given by the absorption argument for
$\mathbb{E}\left(  \sup\nolimits_{r\in\left[  0,t\right]  }\left\vert
P_{r}^{\varepsilon}\right\vert ^{2}\right)  $. From the same equation
(\ref{Eq for apriori estimates}), isolate in the LHS only the integral with
respect to $\left\vert U_{s}^{\varepsilon}\right\vert ^{2}$, apply the
expectation and we get:%
\[%
\begin{array}
[c]{ccl}%
\mathbb{E}\left(
{\displaystyle\int_{0}^{t}}
\left\vert U_{s}^{\varepsilon}\right\vert ^{2}ds\right)  & \leq &
\dfrac{\mathbb{E}\left\vert P_{0}\right\vert ^{2}}{2C_{U}}+\dfrac{C_{f}}%
{C_{U}}\mathbb{E}\left(
{\displaystyle\int_{0}^{t}}
\left\vert f_{s}^{\varepsilon}\right\vert ^{2}ds\right)  +\dfrac{k_{g}^{2}%
}{2C_{U}}\mathbb{E}\left(
{\displaystyle\int_{0}^{t}}
\left\vert \left\vert g_{s}^{\varepsilon}\right\vert \right\vert _{F}%
^{2}ds\right)  +\dfrac{C_{g,4}}{C_{U}}\left(  \mathbb{E}%
{\displaystyle\int_{0}^{t}}
\left\vert \left\vert g_{s}^{\varepsilon}\right\vert \right\vert _{F}%
^{4}ds\right)  \medskip\\
& \leq & \dfrac{\mathbb{E}\left\vert P_{0}\right\vert ^{2}}{2C_{U}}%
+\dfrac{C_{f}}{C_{U}}%
{\displaystyle\int_{0}^{t}}
\left(  f^{\#}\left(  s\right)  \right)  ^{2}ds+\dfrac{k_{g}^{2}}{2C_{U}%
}t^{1/2}\left(
{\displaystyle\int_{0}^{t}}
\left(  g^{\#}\left(  s\right)  \right)  ^{4}ds\right)  ^{1/2}\medskip\\
& + & \dfrac{C_{g,4}}{C_{U}}%
{\displaystyle\int_{0}^{t}}
\left(  g^{\#}\left(  s\right)  \right)  ^{4}ds.
\end{array}
\]
since $f$ and $g$ are bounded on the bounded convex set $C$. The hypothesis on
$P_{0},f$ and $g$ assure the existence of a positive, $\varepsilon
-$independent, constant $C_{1,T}$ such that%
\begin{equation}
\fbox{$\sup\limits_{0<\varepsilon\leq\varepsilon_{0}}\mathbb{E}\left(
{\displaystyle\int_{0}^{T}}
\left\vert U_{s}^{\varepsilon}\right\vert ^{2}ds\right)  <\mathcal{C}_{1,T}$,}
\label{bound of EintUeps2}%
\end{equation}
which is a uniform estimation of $U^{\varepsilon}$ in $L^{2}\left(
\Omega\times\left[  0,T\right]  \right)  $. It follows, since $K_{t}%
^{\varepsilon}=%
{\textstyle\int_{0}^{t}}
\nabla\varphi_{\varepsilon}\left(  X_{s}^{\varepsilon}\right)  ds=%
{\textstyle\int_{0}^{t}}
U_{s}^{\varepsilon}ds$,%
\[
\fbox{$\mathbb{E}\left\updownarrow K^{\varepsilon}\right\updownarrow _{T}\leq
T\mathbb{E}\left(
{\displaystyle\int_{0}^{T}}
\left\vert U_{s}^{\varepsilon}\right\vert ^{2}ds\right)  <T\mathcal{C}_{1,T}%
$,\quad$\forall0<\varepsilon\leq\varepsilon_{0}$~.}%
\]
Obviously, inserting (\ref{bound of EintUeps2}) into
(\ref{prebound of Efieps,EPeps}), we obtain, for all $\varepsilon$,%
\begin{equation}
\fbox{$\mathbb{E}\left(  \sup\limits_{0\leq t\leq T}\varphi_{\varepsilon
}\left(  X_{r}^{\varepsilon}\right)  \right)  +\mathbb{E}\left(
\sup\limits_{0\leq t\leq T}\left\vert P_{r}^{\varepsilon}\right\vert
^{2}\right)  \leq\mathcal{C}_{2,T}$} \label{bound of Efieps,EPeps}%
\end{equation}
Finally, from the first approximated equation of
(\ref{penalized Phase-Space system}) it easily follows:%
\begin{equation}
\fbox{$\sup\limits_{0<\varepsilon\leq\varepsilon_{0}}\mathbb{E}\left(
\sup\limits_{0\leq t\leq T}\left\vert X_{t}^{\varepsilon}\right\vert
^{2}\right)  <\mathcal{C}_{3,T}$.} \label{bound of Xeps}%
\end{equation}

\hfill
\end{proof}

\begin{remark}
For polyhedral sets we obtain the same type of energetic estimate, with the
mention that we replace the Moreau-Yosida penalization based on the distance
with the penalization on faces. We use $\Psi_{\varepsilon}\left(  x\right)
=\tfrac{1}{2\varepsilon}%
{\textstyle\sum\nolimits_{i=1}^{N}}
\left(  l_{i}\left(  x\right)  ^{+}\right)  ^{2}$ introduced in
(\ref{penalization for polyhedral sets}) and the technique is similar. We only
need to impose a non-degeneracy condition of active norms: for a point $x\in
C$, denote $J\left(  x\right)  :=\{i=\overline{1,N}:l_{i}\left(  x\right)
>0\}.$ The set $J$ aims locally, to the points' neighborhood. We assume there
exists $k_{C}>0$ such that
\[
k_{C}\sum_{i\in J\left(  x\right)  }a_{i}^{2}\leq\left\vert \sum_{i\in
J\left(  x\right)  }a_{i}n_{i}\right\vert ^{2},
\]
for all the coefficients $a_{i}\geq0$ which appear simultaneous in the active
region. More precisely, it can be rigorously formalized as follows: for any
possible set $J$ of active faces of a given point, there exists $k_{J}>0$ such
that:%
\begin{equation}
k_{J}\sum_{i\in J}a_{i}^{2}\leq\left\vert \sum_{i\in J}a_{i}n_{i}\right\vert
^{2},\quad a_{i}\geq0, \label{constraint for poly sets}%
\end{equation}
and then take $k_{C}:=\min_{J}k_{J}~,$ value which exists since there is a
finite number of faces and subsets $J$. Usually, in a face neighborhood,
$\mathrm{card}\left(  J\right)  =1,$ in a neighborhood of an edge or vertex,
$J$ can contain more incident faces, so $\mathrm{card}\left(  J\right)  \geq2$.

The above non-degeneracy condition (\ref{constraint for poly sets}) is useful
for controlling the Hessian term via $\left\vert U^{\varepsilon}\right\vert
^{2}$. It is easy verified by a broad range of domains, like simplicial
polyhedra (polyhedra with all faces simplexes).
\end{remark}

\subsection{Tightness via the Aldous' criterion}

\begin{proof}
[Proof (continued)]We first work in the Skorokhod space $\mathcal{D}\left(
\left[  0,T\right]  ;\mathbb{R}^{3\ast2}\right)  $, endowed with the $J_{1}%
$-topology. The maximal increment estimate obtained below, combined with
compact containment, yields $C-$tightness of the family $\{\mathcal{L}\left(
Z^{\varepsilon}\right)  \}_{\varepsilon>0}=\{\mathcal{L}\left(  X^{\varepsilon
},P^{\varepsilon},K^{\varepsilon}\right)  \}_{\varepsilon>0}.$ Consequently,
every weak limit point is concentrated on $C\left(  \left[  0,T\right]
;\mathbb{R}^{3\ast2}\right)  $. Based on (\ref{bound of EintUeps2}),
(\ref{bound of Efieps,EPeps}) and (\ref{bound of Xeps}) we show that we can
apply Aldous tightness criterion for the sequence $Z^{\varepsilon}:=\left(
X^{\varepsilon},P^{\varepsilon},K^{\varepsilon}\right)  .$ The criterion
operates with uniformly estimates of the increments upon stopping times.
Usually, the classical use of this tightness criterion is in the framework of
the space $\mathcal{D}\left(  \left[  0,T\right]  ;\mathbb{R}^{3\ast2}\right)
$ of c\`{a}dl\`{a}g processes, equipped with the Skorokhod topology (see
Billingsley \cite[p.176]{Billingsley:99}) and it provides $\mathcal{D}%
-$tightness. However, the estimates obtained for $Z^{\varepsilon}$, followed
by the continuity of its component limits in%
\[
C(\left[  0,T\right]  ;\mathbb{R}^{3\ast2})\subset\mathcal{D}(\left[
0,T\right]  ;\mathbb{R}^{3\ast2})
\]
can lead to $C-$tightness. Aldous, in its original form from $\mathcal{D}$
requires, for every $\eta>0$,%
\begin{equation}
\lim_{\delta\searrow0}\sup_{0<\varepsilon\leq\varepsilon_{0}}\sup_{\tau\leq
T}\sup_{0\leq\theta\leq\delta}\mathbb{P}\left(  \left\vert Z_{\left(
\tau+\theta\right)  \wedge T}^{\varepsilon}-Z_{\tau}^{\varepsilon}\right\vert
_{\left(  \mathbb{R}^{2}\right)  ^{3}}>\eta\right)  =0,
\label{initial Aldous criterion}%
\end{equation}
where $\tau\in\mathcal{T}_{\left[  0,T\right]  }$ is any stopping time. In
order to prove (\ref{initial Aldous criterion}) we will invoke the Markov's inequality.

Let we recall the phase-space system%
\[
\left\{
\begin{array}
[c]{l}%
dX_{t}^{\varepsilon}=\left(  f\left(  t,X_{t}^{\varepsilon}\right)
-S_{t}U_{t}^{\varepsilon}-G_{t}P_{t}^{\varepsilon}\right)  dt+g\left(
t,X_{t}^{\varepsilon}\right)  dW_{t},\medskip\\
dP_{t}^{\varepsilon}=\left(  f_{phase}\left(  t,X_{t}^{\varepsilon}%
,P_{t}^{\varepsilon}\right)  -G_{t}U_{t}^{\varepsilon}-\Gamma_{t}%
P_{t}^{\varepsilon}\right)  dt+g_{phase}\left(  t,X_{t}^{\varepsilon}%
,P_{t}^{\varepsilon}\right)  dW_{t},\medskip\\
U_{t}^{\varepsilon}:=\nabla\varphi_{\varepsilon}\left(  X_{t}^{\varepsilon
}\right)  ,\quad\mathbb{P-}\text{a.e.,}\quad K_{t}^{\varepsilon}=%
{\displaystyle\int_{0}^{t}}
U_{s}^{\varepsilon}ds.
\end{array}
\right.
\]
According to the hypothesis on $S,G,\Gamma,f,g,f_{phase}$ and $g_{phase},$ the
estimates (\ref{bound of EintUeps2}), (\ref{bound of Efieps,EPeps}) and
(\ref{bound of Xeps}) lead to the existence of some $\varepsilon-$independent
positive constants $C_{T,}$ which can vary from one line to another, such that%
\[
\left\{
\begin{array}
[c]{l}%
\sup\limits_{0<\varepsilon\leq\varepsilon_{0}}\left(  \mathbb{E}%
{\displaystyle\int_{0}^{T}}
\left\vert f\left(  t,X_{t}^{\varepsilon}\right)  \right\vert ^{2}%
dt+\mathbb{E}%
{\displaystyle\int_{0}^{T}}
\left\vert |g\left(  t,X_{t}^{\varepsilon}\right)  \right\vert |_{F}%
^{2}dt\right)  \leq C_{T}~,\medskip\\
\sup\limits_{0<\varepsilon\leq\varepsilon_{0}}\left(  \mathbb{E}%
{\displaystyle\int_{0}^{T}}
\left\vert f_{phase}\left(  t,X_{t}^{\varepsilon},P_{t}^{\varepsilon}\right)
\right\vert ^{2}dt+\mathbb{E}%
{\displaystyle\int_{0}^{T}}
\left\vert |g_{phase}\left(  t,X_{t}^{\varepsilon},P_{t}^{\varepsilon}\right)
\right\vert |_{F}^{2}dt\right)  \leq C_{T}~.
\end{array}
\right.
\]
By denoting%
\[
\left\{
\begin{array}
[c]{l}%
b_{t}^{\varepsilon}:=f\left(  t,X_{t}^{\varepsilon}\right)  -S_{t}%
U_{t}^{\varepsilon}-G_{t}P_{t}^{\varepsilon}~,\medskip\\
b_{phase,t}^{\varepsilon}:=f_{phase}\left(  t,X_{t}^{\varepsilon}%
,P_{t}^{\varepsilon}\right)  -G_{t}U_{t}^{\varepsilon}-\Gamma_{t}%
P_{t}^{\varepsilon}\,\text{,}%
\end{array}
\right.
\]
one obtain, through the already obtained estimates,%
\[
\sup\limits_{0<\varepsilon\leq\varepsilon_{0}}\left(  \mathbb{E}%
{\displaystyle\int_{0}^{T}}
\left\vert b_{t}^{\varepsilon}\right\vert ^{2}dt+\mathbb{E}%
{\displaystyle\int_{0}^{T}}
\left\vert b_{phase,t}^{\varepsilon}\right\vert ^{2}dt\right)  \leq C_{T}%
\]
We focus now on Aldous condition, in its maximal form. Take an arbitrary
chosen stopping time $\tau\in\mathcal{T}_{\left[  0,T\right]  }$,
$0<\delta\leq T$ and denote%
\[
\sigma_{r}:=\left(  \tau+r\right)  \wedge T,\quad\quad0\leq r\leq\delta.
\]
We intend to prove a stronger Aldous' condition, which implies
(\ref{initial Aldous criterion}). More precisely, we first prove that%
\[
\sup_{0<\varepsilon\leq\varepsilon_{0}}\sup_{\tau\leq T}\mathbb{E}\left(
\sup_{0\leq r\leq\delta}\left\vert Z_{\sigma_{r}}^{\varepsilon}-Z_{\tau
}^{\varepsilon}\right\vert ^{2}\right)  \leq C_{T}\delta.
\]

\noindent For $X^{\varepsilon}:$%
\[
X_{\sigma_{r}}^{\varepsilon}-X_{\tau}^{\varepsilon}=\int_{\tau}^{\sigma_{r}%
}b_{s}^{\varepsilon}ds+\int_{\tau}^{\sigma_{r}}g_{s}^{\varepsilon}%
dW_{s}=:A_{r}^{\varepsilon}+M_{r}^{\varepsilon}\,,
\]
where%
\[
\left\{
\begin{array}
[c]{l}%
\sup\limits_{0\leq r\leq\delta}\left\vert A_{r}^{\varepsilon}\right\vert
^{2}=\sup\limits_{0\leq r\leq\delta}\left\vert
{\displaystyle\int_{\tau}^{\sigma_{r}}}
b_{s}^{\varepsilon}ds\right\vert ^{2}\leq\sup\limits_{0\leq r\leq\delta
}\left(  \left(  \sigma_{r}-\tau\right)
{\displaystyle\int_{\tau}^{\sigma_{r}}}
\left\vert b_{s}^{\varepsilon}\right\vert ^{2}ds\right)  \leq\delta%
{\displaystyle\int_{0}^{T}}
\left\vert b_{s}^{\varepsilon}\right\vert ^{2}ds\medskip\\
\sup\limits_{0<\varepsilon\leq\varepsilon_{0}}\sup\limits_{\tau\leq
T}\mathbb{E}\left(  \sup\limits_{0\leq r\leq\delta}\left\vert A_{r}%
^{\varepsilon}\right\vert ^{2}\right)  \leq C_{T}\delta,\medskip\\
\mathbb{E}\left(  \sup\limits_{0\leq r\leq\delta}\left\vert M_{r}%
^{\varepsilon}\right\vert ^{2}\right)  \overset{BDG~inequality}{\leq
}C\mathbb{E}%
{\displaystyle\int_{\tau}^{\left(  \tau+\delta\right)  \wedge T}}
\left\vert \left\vert g_{s}^{\varepsilon}\right\vert \right\vert _{F}%
^{2}ds\leq C_{T}\delta.
\end{array}
\right.
\]
Consequently,%
\[
\mathbb{E}\left(  \sup\limits_{0\leq r\leq\delta}\left\vert X_{\sigma_{r}%
}^{\varepsilon}-X_{\tau}^{\varepsilon}\right\vert ^{2}\right)  \leq
2\mathbb{E}\left(  \sup\limits_{0\leq r\leq\delta}\left\vert A_{r}%
^{\varepsilon}\right\vert ^{2}\right)  +2\mathbb{E}\left(  \sup\limits_{0\leq
r\leq\delta}\left\vert M_{r}^{\varepsilon}\right\vert ^{2}\right)  \leq
C_{T}\delta.
\]

\noindent For $P^{\varepsilon}:\quad$Start we the phase equation and, exactly
with the same arguments as above, we obtain:%
\[
\sup\limits_{0<\varepsilon\leq\varepsilon_{0}}\sup\limits_{\tau\leq
T}\mathbb{E}\left(  \sup\limits_{0\leq r\leq\delta}\left\vert P_{\sigma_{r}%
}^{\varepsilon}-P_{\tau}^{\varepsilon}\right\vert ^{2}\right)  \leq
C_{T}\delta.
\]

\noindent For $K^{\varepsilon}:\quad$We have, by the Cauchy-Schwarz
inequality,%
\[
\sup\limits_{0<\varepsilon\leq\varepsilon_{0}}\sup\limits_{\tau\leq
T}\mathbb{E}\left(  \sup\limits_{0\leq r\leq\delta}\left\vert K_{\sigma_{r}%
}^{\varepsilon}-K_{\tau}^{\varepsilon}\right\vert ^{2}\right)  \leq\delta
\sup\limits_{0<\varepsilon\leq\varepsilon_{0}}\sup\limits_{\tau\leq
T}\mathbb{E}%
{\displaystyle\int_{0}^{T}}
\left\vert U_{s}^{\varepsilon}\right\vert ^{2}ds\leq C_{T}\delta.
\]
Putting together the last three estimates, we get%
\begin{equation}
\sup\limits_{0<\varepsilon\leq\varepsilon_{0}}\sup\limits_{\tau\leq
T}\mathbb{E}\left(  \sup_{0\leq r\leq\delta}\left\vert Z_{\left(
\tau+r\right)  \wedge T}^{\varepsilon}-Z_{\tau}^{\varepsilon}\right\vert
^{2}\right)  \leq C_{T}\delta. \label{boundedness of Zeps}%
\end{equation}
We use now the Markov inequality, but first we have to make an intermediary
step in order to apply it correctly. Denote $Y_{\varepsilon,\tau,\delta}%
:=\sup_{0\leq r\leq\delta}\left\vert Z_{\left(  \tau+r\right)  \wedge
T}^{\varepsilon}-Z_{\tau}^{\varepsilon}\right\vert ^{2}$. Since
$Y_{\varepsilon,\tau,\delta}\geq0,$ $\mathbb{P-}$a.e.,%
\[
\left\{  \sup_{0\leq r\leq\delta}\left\vert Z_{\left(  \tau+r\right)  \wedge
T}^{\varepsilon}-Z_{\tau}^{\varepsilon}\right\vert >\eta\right\}  =\left\{
Y_{\varepsilon,\tau,\delta}>\eta^{2}\right\}  .
\]
By the Markov inequality we obtain%
\[
\mathbb{P}\left(  \sup_{0\leq r\leq\delta}\left\vert Z_{\left(  \tau+r\right)
\wedge T}^{\varepsilon}-Z_{\tau}^{\varepsilon}\right\vert >\eta\right)
=\mathbb{P}\left(  Y_{\varepsilon,\tau,\delta}>\eta^{2}\right)  \leq
\frac{\mathbb{E}Y_{\varepsilon,\tau,\delta}}{\eta^{2}}\leq\frac{C_{T}\delta
}{\eta^{2}}.
\]
As consequence, we obtain the desired maximal form of Aldous' condition:%
\begin{equation}
\lim_{\delta\searrow0}\sup_{0<\varepsilon\leq\varepsilon_{0}}\sup_{\tau\leq
T}\mathbb{P}\left(  \sup_{0\leq\theta\leq\delta}\left\vert Z_{\left(
\tau+\theta\right)  \wedge T}^{\varepsilon}-Z_{\tau}^{\varepsilon}\right\vert
>\eta\right)  =0, \label{Aldous maximal}%
\end{equation}
which is stronger that the initial form (\ref{initial Aldous criterion}). This
stronger condition is exactly the quantity that appears when we want to
exclude aggressive oscillations on small time intervals, fact that allows us
to derive $C-$tightness, i.e. tightness with continuous limits. This maximal
form is the suitable one to exclude limit jumps (see Aldous \cite{Aldous:89}).
In order to observe the condition for two near moments, let $0\leq r_{1}%
<r_{2}\leq\delta$. Hence%
\[
\left\vert Z_{\left(  \tau+r_{2}\right)  \wedge T}^{\varepsilon}-Z_{\left(
\tau+r_{1}\right)  \wedge T}^{\varepsilon}\right\vert \leq\left\vert
Z_{\left(  \tau+r_{2}\right)  \wedge T}^{\varepsilon}-Z_{\tau}^{\varepsilon
}\right\vert +\left\vert Z_{\left(  \tau+r_{1}\right)  \wedge T}^{\varepsilon
}-Z_{\tau}^{\varepsilon}\right\vert \leq2\sup_{0\leq r\leq\delta}\left\vert
Z_{\left(  \tau+r\right)  \wedge T}^{\varepsilon}-Z_{\tau}^{\varepsilon
}\right\vert
\]
and (\ref{Aldous maximal}) also implies%
\[
\lim_{\delta\searrow0}\sup_{0<\varepsilon\leq\varepsilon_{0}}\sup_{\tau\leq
T}\mathbb{P}\left(  \sup_{0\leq r_{1},r_{2}\leq\delta}\left\vert Z_{\left(
\tau+r_{2}\right)  \wedge T}^{\varepsilon}-Z_{\left(  \tau+r_{1}\right)
\wedge T}^{\varepsilon}\right\vert >\eta\right)  =0.
\]
To conclude the Aldous' tightness criterion, we have to verify the compact
containment. The energy estimates assure
\[
\mathbb{E}\left\updownarrow K^{\varepsilon}\right\updownarrow _{T}%
^{2}=\mathbb{E}\left(
{\displaystyle\int_{0}^{T}}
\left\vert U_{s}^{\varepsilon}\right\vert ds\right)  ^{2}\leq T\mathbb{E}%
\left(
{\displaystyle\int_{0}^{T}}
\left\vert U_{s}^{\varepsilon}\right\vert ^{2}ds\right)  \leq C_{T}.
\]
Together with the estimates of $X^{\varepsilon}$ and $P^{\varepsilon}$ we
obtain $\sup_{0<\varepsilon\leq\varepsilon_{0}}\mathbb{E}\left\vert \left\vert
Z^{\varepsilon}\right\vert \right\vert _{\infty}^{2}\leq C_{T}$ and the Markov
inequality gives%
\[
\sup_{0<\varepsilon\leq\varepsilon_{0}}\mathbb{P}\left(  \sup_{0\leq t\leq
T}\left\vert Z_{t}^{\varepsilon}\right\vert >R\right)  =\sup_{0<\varepsilon
\leq\varepsilon_{0}}\mathbb{P}\left(  \left\vert \left\vert Z^{\varepsilon
}\right\vert \right\vert _{\infty}>R\right)  \leq\frac{C_{T}}{R^{2}%
}\rightarrow0,\quad\text{as}\quad R\rightarrow+\infty\text{.}%
\]
Consequently, the maximal Aldous estimate, together with the above compact
containment estimate, yields%
\[
\left\{  \mathcal{L}\left(  Z^{\varepsilon}\right)  \right\}  _{\varepsilon
>0}=\left\{  \mathcal{L}\left(  X^{\varepsilon},P^{\varepsilon},K^{\varepsilon
}\right)  \right\}  _{\varepsilon>0}\quad\text{is}\quad C-\text{tight in
}\mathcal{D}(\left[  0,T\right]  ;\mathbb{R}^{3\ast2})\text{.}%
\]
Since the approximating processes have continuous trajectories, every weak
limit point in the Skorokhod topology is concentrated on $C(\left[
0,T\right]  ;\mathbb{R}^{3\ast2})$, hence the sequence is tight when regarded
as a family of laws on the continuous-path space $C(\left[  0,T\right]
;\mathbb{R}^{3\ast2})$. Moreover, the total variations of the family
$\{K^{\varepsilon}\}_{\varepsilon}$ are uniformly bounded in $L^{2}$ and we
can apply the same tightness arguments for $V^{\varepsilon}:=\left\updownarrow
K^{\varepsilon}\right\updownarrow $. Since the Wiener component has a fixed
Wiener law, by the Prohorov theorem, there exists a subsequence such that, as
$\varepsilon_{n}\rightarrow0$,
\[
\left(  X^{\varepsilon_{n}},P^{\varepsilon_{n}},K^{\varepsilon_{n}%
},V^{\varepsilon_{n}},W\right)  \rightarrow\left(  X,P,K,V,W\right)
,\quad\text{in law on }C\left(  \left[  0,T\right]  ;\mathbb{R}^{3\ast
2+1+k}\right)  .
\]
By the Skorokhod representation theorem, we can choose a probability space
$\left(  \bar{\Omega},\mathcal{\bar{F}},\mathbb{\bar{P}}\right)  $ and some
random vectors defined on $\left(  \bar{\Omega},\mathcal{\bar{F}}%
,\mathbb{\bar{P}}\right)  $,%
\[
\left\{
\begin{array}
[c]{l}%
\left(  \bar{X}^{\varepsilon_{n}},\bar{P}^{\varepsilon_{n}},\bar
{K}^{\varepsilon_{n}},\bar{V}^{\varepsilon_{n}},\bar{W}^{\varepsilon_{n}%
}\right)  \overset{\mathcal{L}}{=}\left(  X^{\varepsilon_{n}},P^{\varepsilon
_{n}},K^{\varepsilon_{n}},V^{\varepsilon_{n}},W\right)  \medskip\\
\left(  \bar{X},\bar{P},\bar{K},\bar{V},\bar{W}\right)  \overset{\mathcal{L}%
}{=}\left(  X,P,K,V,W\right)
\end{array}
\right.
\]
such that $\left(  \bar{X}^{\varepsilon_{n}},\bar{P}^{\varepsilon_{n}},\bar
{K}^{\varepsilon_{n}},\bar{V}^{\varepsilon_{n}},\bar{W}^{\varepsilon_{n}%
}\right)  \overset{\mathbb{P}-a.s.}{\rightarrow}\left(  \bar{X},\bar{P}%
,\bar{K},\bar{V},\bar{W}\right)  .$

Moreover, according to Pardoux, R\u{a}\c{s}canu \cite[Proposition
2.15.]{Pardoux/Rascanu:14}, $(\bar{W}^{\varepsilon_{n}},\mathbb{F}^{\bar
{X}^{\varepsilon_{n}},\bar{P}^{\varepsilon_{n}},\bar{K}^{\varepsilon_{n}}%
,\bar{V}^{\varepsilon_{n}},\bar{W}^{\varepsilon_{n}}})_{n}$ and $(\bar
{W},\mathbb{F}^{\bar{X},\bar{P},\bar{K},\bar{V},\bar{W}})$ are $\mathbb{R}%
^{k}-$Wiener processes. Let us summarize that, along a further subsequence, we
may assume almost sure convergence in $\mathcal{D}(\left[  0,T\right]
;\mathbb{R}^{3\ast2+1+k})$. Since the limit is continuous, this convergence is
uniformly on $\left[  0,T\right]  .$ The lower semicontinuity of the total
variation gives $\left\updownarrow \bar{K}\right\updownarrow _{T}\leq
\liminf_{n\rightarrow+\infty}\bar{V}_{T}^{\varepsilon_{n}}=\bar{V}_{T},$
$\mathbb{\bar{P}-}$a.e.. According to the Fatou's Lemma,%
\[
\mathbb{\bar{E}}\left\updownarrow \bar{K}\right\updownarrow _{T}^{2}%
\leq\mathbb{\bar{E}}\bar{V}_{T}^{2}\overset{Fatou}{\leq}\liminf_{n\rightarrow
+\infty}\mathbb{\bar{E}}\left(  \bar{V}_{T}^{\varepsilon_{n}}\right)
^{2}=\liminf_{n\rightarrow+\infty}\mathbb{\bar{E}}\left\updownarrow \bar
{K}^{\varepsilon_{n}}\right\updownarrow _{T}^{2}<+\infty.
\]
Thus, along a suitable sequence, there exists the limit $\bar{Z}:=\left(
\bar{X},\bar{P},\bar{K}\right)  $ of the penalizing sequence of "copies" given
by the Skorokhod representation $\bar{Z}^{\varepsilon}=\left(  \bar
{X}^{\varepsilon},\bar{P}^{\varepsilon},\bar{K}^{\varepsilon}\right)
_{\varepsilon}$ such that%
\[
\left\{
\begin{array}
[c]{l}%
\left(  \bar{X}_{\left(  \cdot\right)  },\bar{P}_{\left(  \cdot\right)
}\right)  \in C\left(  \left[  0,T\right]  ;\mathbb{R}^{2\ast2}\right)
,\text{ }\mathbb{P-}\text{a.s.}\quad\text{and}\quad\mathbb{\bar{E}}\left(
\sup\nolimits_{t\in\left[  0,T\right]  }\left\vert \bar{X}_{t}\right\vert
^{2}+\sup\nolimits_{t\in\left[  0,T\right]  }\left\vert \bar{P}_{t}\right\vert
^{2}\right)  <+\infty,\quad i.e.\medskip\\
\left(  \bar{X},\bar{P}\right)  \in S_{2\ast2}^{2}\left[  0,T\right]
,\medskip\\
\bar{K}_{\left(  \cdot\right)  }\in C\left(  \left[  0,T\right]
;\mathbb{R}^{2}\right)  \cap BV\left(  \left[  0,T\right]  ;\mathbb{R}%
^{2}\right)  ,\quad\mathbb{\bar{P}-}\text{a.s.}\quad\text{and}\quad
\mathbb{\bar{E}}\left\updownarrow \bar{K}\right\updownarrow _{T}^{2}<+\infty.
\end{array}
\right.
\]

\hfill
\end{proof}

\subsection{Identification of a solution}

\begin{proof}
[Proof (continued)]The last milestone of our proof is to show that the triplet
$\left(  \bar{X},\bar{P},\bar{K}\right)  $ is a solution in the sense of
Definition \ref{Definition of Solution}. On the Skorokhod probability space,
the almost sure convergence of $\left(  \bar{X}^{\varepsilon_{n}},\bar
{P}^{\varepsilon_{n}},\bar{K}^{\varepsilon_{n}}\right)  $ to $\left(  \bar
{X},\bar{P},\bar{K}\right)  $, together with the continuity of the
coefficients and the boundedness of $S,G,\Gamma$, allows us to pass to the
limit in all integrals with deterministic integrators, of Riemann or Stieltjes
type. The convergence of the stochastic integrals follows from the martingale
characterization and the fact that $\bar{W}$ is a Wiener process with respect
to the natural filtration of the limit. The subdifferential condition is
obtained by passing to the limit in the Yosida approximation inequality. We
present below the main arguments for our present statements.

Let us first remind the fact that, in the previous section, we obtained the
almost sure convergence of $\left(  \bar{X}^{\varepsilon_{n}},\bar
{P}^{\varepsilon_{n}},\bar{K}^{\varepsilon_{n}}\right)  $ to $\left(  \bar
{X},\bar{P},\bar{K}\right)  $ in $\mathcal{D}(\left[  0,T\right]
;\mathbb{R}^{3\ast2})$. However, if we have this type of convergence, with
respect to the $J_{1}$-topology on $\mathcal{D}(\left[  0,T\right]
;\mathbb{R}^{3\ast2})$ and the limits are continuous, we have, in fact uniform
almost sure convergence:%
\begin{equation}
\sup_{0\leq t\leq T}\left(  \left\vert \bar{X}_{t}^{\varepsilon_{n}}-\bar
{X}_{t}\right\vert +\left\vert \bar{P}_{t}^{\varepsilon_{n}}-\bar{P}%
_{t}\right\vert \right)  \rightarrow0,\quad\text{as}\quad\varepsilon
_{n}\searrow0. \label{convergence of Xeps, Peps}%
\end{equation}
Now, since $f,g,f_{phase},g_{phase}$ are Carath\'{e}odory functions, yields,
for every $t$, $\mathbb{\bar{P}-}$a.s.%
\[
\left\{
\begin{array}
[c]{l}%
\left(  f\left(  s,\bar{X}_{s}^{\varepsilon_{n}}\right)  ,g\left(  s,\bar
{X}_{s}^{\varepsilon_{n}}\right)  \right)  \rightarrow\left(  f\left(
s,\bar{X}_{s}\right)  ,g\left(  s,\bar{X}_{s}\right)  \right)  ,\medskip\\
\left(  f_{phase}\left(  s,\bar{X}_{s}^{\varepsilon_{n}},\bar{P}%
^{\varepsilon_{n}}\right)  ,g_{phase}\left(  s,\bar{X}_{s}^{\varepsilon_{n}%
},\bar{P}^{\varepsilon_{n}}\right)  \right)  \rightarrow\left(  f_{phase}%
\left(  s,\bar{X}_{s},\bar{P}_{s}\right)  ,g_{phase}\left(  s,\bar{X}_{s}%
,\bar{P}_{s}\right)  \right)  .
\end{array}
\right.
\]
Also, using $\left(  H_{f,g}\right)  $ we obtain, for $x_{0}\in Dom\left(
\varphi\right)  $,%
\[
\left\vert f\left(  s,\bar{X}_{s}^{\varepsilon_{n}}\right)  \right\vert
^{2}\leq2|f^{\#}\left(  s\right)  |^{2}+2\mu^{2}\left(  s\right)  \left\vert
\bar{X}_{s}^{\varepsilon_{n}}-x_{0}\right\vert ^{2},\quad\mathbb{\bar{P}%
-}\text{a.s}%
\]
and%
\[
\sup_{n}\mathbb{\bar{E}}%
{\displaystyle\int_{0}^{T}}
\left\vert f\left(  s,\bar{X}_{s}^{\varepsilon_{n}}\right)  \right\vert
^{2}ds\leq2%
{\displaystyle\int_{0}^{T}}
|f^{\#}\left(  s\right)  |^{2}ds+2L^{2}T\mathcal{C}_{3,T}<+\infty.
\]
This justify the possibility of passing to limit via Lebesgue dominated
convergence theorem:%
\[%
{\displaystyle\int_{0}^{t}}
f\left(  s,\bar{X}_{s}^{\varepsilon_{n}}\right)  ds\rightarrow%
{\displaystyle\int_{0}^{t}}
f\left(  s,\bar{X}_{s}\right)  ds,\quad\text{as }\varepsilon_{n}%
\searrow0\text{,}\quad\mathbb{\bar{P}-}\text{a.s.}%
\]
The same boundedness argument applies to $g,f_{phase}$ and $g_{phase}$.

For the convergence of the integral with the integrand $G_{s}P_{s}%
^{\varepsilon}:$%
\[
\left\{
\begin{array}
[c]{l}%
\left\vert
{\displaystyle\int_{0}^{t}}
G_{s}\bar{P}_{s}^{\varepsilon_{n}}ds-%
{\displaystyle\int_{0}^{t}}
G_{s}\bar{P}_{s}ds\right\vert \leq%
{\displaystyle\int_{0}^{t}}
\left\vert G_{s}\right\vert \left\vert \bar{P}_{s}^{\varepsilon_{n}}-\bar
{P}_{s}\right\vert ds\leq\sup\limits_{0\leq t\leq T}\left\vert G_{t}%
\right\vert
{\displaystyle\int_{0}^{T}}
\left\vert \bar{P}_{s}^{\varepsilon_{n}}-\bar{P}_{s}\right\vert ds\medskip\\
\quad\quad\leq\left\vert \left\vert G\right\vert \right\vert _{\infty}%
{\displaystyle\int_{0}^{T}}
\sup\limits_{0\leq t\leq T}\left\vert \bar{P}_{t}^{\varepsilon_{n}}-\bar
{P}_{t}\right\vert dt\leq\left\vert \left\vert G\right\vert \right\vert
_{\infty}T\sup\limits_{0\leq t\leq T}\left\vert \bar{P}_{t}^{\varepsilon_{n}%
}-\bar{P}_{t}\right\vert \rightarrow0,\quad\text{as}\quad\varepsilon
_{n}\searrow0.\medskip\\
\left\vert
{\displaystyle\int_{0}^{t}}
G_{s}\bar{P}_{s}^{\varepsilon_{n}}ds\right\vert ^{2}\leq\left(
{\displaystyle\int_{0}^{t}}
\left\vert G_{s}\right\vert \left\vert \bar{P}_{s}^{\varepsilon_{n}%
}\right\vert ds\right)  ^{2}\leq t%
{\displaystyle\int_{0}^{t}}
\left\vert G_{s}\right\vert ^{2}\left\vert \bar{P}_{s}^{\varepsilon_{n}%
}\right\vert ^{2}ds\leq T\left\vert \left\vert G\right\vert \right\vert
_{\infty}^{2}%
{\displaystyle\int_{0}^{T}}
\left\vert \bar{P}_{s}^{\varepsilon_{n}}\right\vert ^{2}ds.\medskip\\
\text{So:}\quad\mathbb{\bar{E}}\left(  \left\vert
{\displaystyle\int_{0}^{t}}
G_{s}\bar{P}_{s}^{\varepsilon_{n}}ds\right\vert ^{2}\right)  \leq T\left\vert
\left\vert G\right\vert \right\vert _{\infty}^{2}\sup\limits_{n}%
\mathbb{\bar{E}}\left(  \sup\limits_{0\leq s\leq T}\left\vert \bar{P}%
_{s}^{\varepsilon_{n}}\right\vert ^{2}\right)  <+\infty.
\end{array}
\right.
\]
Hence, the family is uniformly integrable and the almost surely convergence
(with respect to $t$) implies convergence in $L^{1}\left(  \Omega\right)  $.
Exactly the same argument leads us to%
\[%
{\displaystyle\int_{0}^{t}}
\Gamma_{s}\bar{P}_{s}^{\varepsilon_{n}}ds\rightarrow%
{\displaystyle\int_{0}^{t}}
\Gamma_{s}\bar{P}_{s}ds,\quad\text{as }\varepsilon_{n}\searrow0\text{,}%
\quad\text{uniformly w.r.t. }t\in\left[  0,T\right]  ,\quad\mathbb{\bar{P}%
-}\text{a.s.}%
\]

For the convergence of the sequence of Riemann-Stieltjes integrals,%
\[%
{\displaystyle\int_{0}^{t}}
S_{s}U_{s}^{\varepsilon_{n}}ds=%
{\displaystyle\int_{0}^{t}}
S_{s}d\bar{K}_{s}^{\varepsilon_{n}}\rightarrow%
{\displaystyle\int_{0}^{t}}
S_{s}d\bar{K}_{s},\quad\text{as }\varepsilon_{n}\searrow0\text{,}%
\quad\text{where}\quad\bar{K}_{t}^{\varepsilon_{n}}=%
{\displaystyle\int_{0}^{t}}
U_{s}^{\varepsilon_{n}}ds,
\]
we just have to make us of the Helly-Bray theorem, which permits the passing
to limit in Riemann-Stieltjes integrals (see Pardoux, R\u{a}\c{s}canu
\cite[Proposition 6.16, p. 534]{Pardoux/Rascanu:14}). We can do this since
$\bar{K}^{\varepsilon_{n}}\rightarrow\bar{K}$, as $\varepsilon_{n}%
\rightarrow0$, unif. a.s., $\sup_{n}\mathbb{E}\left\updownarrow \bar
{K}^{\varepsilon_{n}}\right\updownarrow _{T}<+\infty$ and $S$ is continuous
and bounded. The same argument is used for the convergence of $%
{\textstyle\int_{0}^{t}}
G_{s}d\bar{K}_{s}^{\varepsilon_{n}}$ to $%
{\textstyle\int_{0}^{t}}
G_{s}d\bar{K}_{s}~$.

For the convergence of the stochastic integrals, we have to prove, uniformly
w.r.t. $t\in\left[  0,T\right]  $,%
\[
M_{t}^{\varepsilon_{n}}:=%
{\displaystyle\int_{0}^{t}}
g\left(  s,\bar{X}_{s}^{\varepsilon_{n}}\right)  d\bar{W}_{s}^{\varepsilon
_{n}}\rightarrow%
{\displaystyle\int_{0}^{t}}
g\left(  s,\bar{X}_{s}\right)  d\bar{W}_{s}=:M_{t}~,\quad\text{as }%
\varepsilon_{n}\searrow0\text{.}%
\]
The interested reader is invited to consult also Pardoux, R\u{a}\c{s}canu
\cite[Corollary 2.14. and Proposition 2.15.]{Pardoux/Rascanu:14}. For every
$n\in\mathbb{N}^{\ast}$, $M^{\varepsilon_{n}}$ is a continuous martingale with
respect to the sequence filtration $\mathbb{F}^{\bar{X}^{\varepsilon_{n}}%
,\bar{P}^{\varepsilon_{n}},\bar{K}^{\varepsilon_{n}},\bar{V}^{\varepsilon_{n}%
},\bar{W}^{\varepsilon_{n}}}$. Moreover,%
\[
\left\langle M^{\varepsilon_{n}}\right\rangle _{t}=%
{\displaystyle\int_{0}^{t}}
g\left(  s,\bar{X}_{s}^{\varepsilon_{n}}\right)  g\left(  s,\bar{X}%
_{s}^{\varepsilon_{n}}\right)  ^{\ast}ds\quad\text{and}\quad\left\langle
M^{\varepsilon_{n}},\bar{W}^{\varepsilon_{n}}\right\rangle _{t}=%
{\displaystyle\int_{0}^{t}}
g\left(  s,\bar{X}_{s}^{\varepsilon_{n}}\right)  ds.
\]
From the convergence (\ref{convergence of Xeps, Peps}), the fact that $g$ is
continuous and the uniform estimates holds,%
\[
\sup_{0\leq t\leq T}\left\vert \left\langle M^{\varepsilon_{n}}\right\rangle
_{t}-%
{\displaystyle\int_{0}^{t}}
g\left(  s,\bar{X}_{s}\right)  g\left(  s,\bar{X}_{s}\right)  ^{\ast
}ds\right\vert \rightarrow0,\quad\text{as }\varepsilon_{n}\searrow
0,\quad\mathbb{\bar{P}-}\text{a.s.}%
\]
and, for every $t\in\left[  0,T\right]  $,%
\[
\left\langle M^{\varepsilon_{n}},\bar{W}^{\varepsilon_{n}}\right\rangle
_{t}\rightarrow%
{\displaystyle\int_{0}^{t}}
g\left(  s,\bar{X}_{s}\right)  ds\quad\text{as }\varepsilon_{n}\searrow
0,\quad\mathbb{\bar{P}-}\text{a.s.}%
\]
So, the limit process $M$ has the quadratic variation%
\begin{equation}
\left\langle M\right\rangle _{t}=%
{\displaystyle\int_{0}^{t}}
g\left(  s,\bar{X}_{s}\right)  g\left(  s,\bar{X}_{s}\right)  ^{\ast}%
ds,\quad\mathbb{\bar{P}-}\text{a.s.} \label{quadratic variation of M}%
\end{equation}
and, according to the result of martingale stability from Gassous,
R\u{a}\c{s}canu, Rotenstein \cite[Appendix A.4, Proposition 17 and Lemma
19]{Gassous/Rascanu/Rotenstein:12}, $M$ is a martingale with respect to the
natural filtration of the limit, $\mathbb{F}^{\bar{X},\bar{P},\bar{K},\bar
{V},\bar{W}}$. Moreover,%
\[
\left\langle M,\bar{W}\right\rangle _{t}=%
{\displaystyle\int_{0}^{t}}
g\left(  s,X_{s}\right)  ds.
\]
As consequence, from the form of its quadratic variation
(\ref{quadratic variation of M}) and from the fact that $\bar{W}$ is an
$\mathbb{F}^{\bar{X},\bar{P},\bar{K},\bar{V},\bar{W}}-$Wiener process, we
deduce the identification:%
\[
M_{t}=%
{\displaystyle\int_{0}^{t}}
g\left(  s,\bar{X}_{s}\right)  d\bar{W}_{s}~.
\]
The same arguments will assure the convergence of the It\^{o} integrals
featuring the integrand $g_{phase}$~.$\smallskip$

We now have all the ingredients which allow to pass to the limit, as
$\varepsilon_{n}\searrow0$, in the Phase-Space approximating system
(\ref{Phase-Space system suitable}). We get:%
\[
\left\{
\begin{array}
[c]{l}%
\bar{X}_{t}+%
{\displaystyle\int_{0}^{t}}
S_{s}d\bar{K}_{s}+%
{\displaystyle\int_{0}^{t}}
G_{s}\bar{P}_{s}ds=\bar{X}_{0}+%
{\displaystyle\int_{0}^{t}}
f\left(  s,\bar{X}_{s}\right)  ds+%
{\displaystyle\int_{0}^{t}}
g\left(  s,\bar{X}_{s}\right)  d\bar{W}_{s},\medskip\\
\bar{P}_{t}+%
{\displaystyle\int_{0}^{t}}
G_{s}d\bar{K}_{s}+%
{\displaystyle\int_{0}^{t}}
\Gamma_{s}\bar{P}_{s}ds=\bar{P}_{0}+%
{\displaystyle\int_{0}^{t}}
f_{phase}\left(  s,\bar{X}_{s},\bar{P}_{s}\right)  ds+%
{\displaystyle\int_{0}^{t}}
g_{phase}\left(  s,\bar{X}_{s},\bar{P}_{s}\right)  d\bar{W}_{s}~.
\end{array}
\right.
\]
To conclude the proof, we have to identify the subdifferential inclusion, i.e.
we need to obtain $d\bar{K}_{t}\in\partial\varphi\left(  \bar{X}_{t}\right)
\left(  dt\right)  $, for all $t$. We are forced to work with measures and the
integral form of the subdifferential condition for the subdifferential
operator since, on the original space, $U_{t}^{\varepsilon_{n}}=\nabla
\varphi_{\varepsilon_{n}}\left(  X_{t}^{\varepsilon_{n}}\right)  ,$ for all
$t$, but passing into the Skorokhod space, $U_{t}^{\varepsilon_{n}}$ is lost.
By the Skorokhod construction, $\mathcal{L}\left(  K^{\varepsilon_{n}}\right)
=\mathcal{L}\left(  \bar{K}^{\varepsilon_{n}}\right)  $, so, as a copy of
$K^{\varepsilon_{n}}$ - which satisfy an integral type subdifferential
inclusion on the original space - $\bar{K}^{\varepsilon_{n}}$ inherits the
inequality. More precisely, for all $0\leq s\leq t\leq T$:%
\[%
{\displaystyle\int_{s}^{t}}
\varphi\left(  \bar{X}_{r}^{e_{n}}\right)  dr\leq%
{\displaystyle\int_{s}^{t}}
\varphi\left(  y_{r}\right)  dr-%
{\displaystyle\int_{s}^{t}}
\left\langle y_{r}-\bar{X}_{r}^{e_{n}},d\bar{K}_{r}^{\varepsilon_{n}%
}\right\rangle \quad\forall y\in C(\left[  0,T\right]  ;\mathbb{R}^{2}%
),\quad\mathbb{\bar{P}-}\text{a.s.}%
\]
According to Pardoux, R\u{a}\c{s}canu \cite[Corollary 1.18., Proposition
1.19.]{Pardoux/Rascanu:14}%
\[%
{\displaystyle\int_{s}^{t}}
\varphi\left(  \bar{X}_{r}\right)  dr\leq%
{\displaystyle\int_{s}^{t}}
\varphi\left(  y_{r}\right)  dr-%
{\displaystyle\int_{s}^{t}}
\left\langle y_{r}-\bar{X}_{r},d\bar{K}_{r}\right\rangle \quad\forall y\in
C(\left[  0,T\right]  ;\mathbb{R}^{2}),\quad\mathbb{\bar{P}-}\text{a.s.,}%
\]
which implies (see Pardoux, R\u{a}\c{s}canu \cite[Lemma 4.31.]%
{Pardoux/Rascanu:14}) $d\bar{K}_{r}\in\partial\varphi\left(  \bar{X}%
_{r}\right)  \left(  dr\right)  $, for all $r\in\left[  0,T\right]  $.\medskip

As consequence, $(\bar{\Omega},\mathcal{\bar{F}},\mathbb{\bar{P}}%
,\mathbb{F}^{\bar{X},\bar{P},\bar{K},\bar{W}},\bar{W},\bar{X},\bar{P},\bar
{K})$ is a weak-strong solution for the Stochastic Variational Inequality with
oblique reflection (\ref{Phase-Space ssytem}). The existence result is now
complete.\hfill
\end{proof}

\section{On the uniqueness of the solution}

The present construction yields a weak-strong existence result for the lifted
phase-space system. Since the solution is obtained after a compactness
argument and a Skorokhod representation, the theorem does not place arbitrary
solutions on a common stochastic basis or under a common Wiener process.
Consequently, several standard notions of "uniqueness" must be distinguished
carefully: pathwise uniqueness, uniqueness in law independent on the filtered
probability space, solution producing a minimum energy.

The usual \textit{pathwise uniqueness} property cannot be considered at that
level by deriving directly, from the present existence theorem. This type of
uniqueness, together with a \textit{weak existence}, usually assures, via
Ikeda, Watanabe \cite{Ikeda/Watanabe:81} or, more suitable for our framework,
Kurtz \cite{Kurtz:07}, \cite{Kurtz:13}, the existence of a strong solution.
However, for the arguments to hold, one must prove, prior to the existence
result, the pathwise uniqueness for two solutions, synchronized on the same
filtered probability space, with the same initial condition and driven by the
same Wiener process, without making any reference to the existence result.
Such an analysis requires a phase-space energy compatible both with the
monotonicity of the subdifferential term and with the gyroscopic cancellation
induced by the skew-symmetric component. Once pathwise uniqueness is
available, uniqueness in law and strong existence may be approached through a
Yamada--Watanabe-type principle for stochastic differential inclusions with
bounded-variation reflection. The phase-space lifting suggests that uniqueness
should be investigated through a relative energy on the extended state space.
Such an energy must preserve the cancellation induced by the skew-symmetric
coupling and reduce the bounded-variation contribution to the monotonicity
pairing associated with the subdifferential operator.$\smallskip$

Remind the form of the scaled rotation matrix:%
\[
\alpha_{t}\Theta\left(  \theta_{t}\right)  =\alpha_{t}\cos\theta_{t}%
I_{2}+\alpha_{t}\sin\theta_{t}J,\quad\text{where}\quad I_{2}=\left(
\begin{array}
[c]{cc}%
1 & 0\\
0 & 1
\end{array}
\right)  ,\quad J=\left(
\begin{array}
[c]{cc}%
0 & -1\\
1 & 0
\end{array}
\right)  .
\]
Denote $s_{t}:=\alpha_{t}\cos\theta_{t}>0$ and $g_{t}:=\alpha_{t}\sin
\theta_{t}$. Therefore, $S_{t}=s_{t}I_{2}$ and $G_{t}=g_{t}J.$ Obviously,
$J^{T}=-J$ and $J^{2}=-I_{2}~.$

\begin{theorem}
Under the above assumptions, pathwise uniqueness holds.
\end{theorem}

\begin{proof}
Consider two solutions defined on the same stochastic basis, equipped with the
same Wiener process, $W:$%
\[
\left(  X^{1},P^{1},K^{1}\right)  \quad\quad\text{and}\quad\quad\left(
X^{2},P^{2},K^{2}\right)
\]
Denote%
\[
\left\{
\begin{array}
[c]{l}%
\left(  \Delta X_{t},\Delta P_{t},\Delta K_{t}\right)  :=\left(  X_{t}%
^{1}-X_{t}^{2},P_{t}^{1}-P_{t}^{2},K_{t}^{1}-K_{t}^{2}\right)  \medskip\\
\left(  \Delta f_{t},\Delta g_{t}\right)  :=\left(  f\left(  t,X_{t}%
^{1}\right)  -f\left(  t,X_{t}^{2}\right)  ,g\left(  t,X_{t}^{1}\right)
-g\left(  t,X_{t}^{2}\right)  \right)  \medskip\\
\Delta f_{phase,t}:=f_{phase}\left(  t,X_{t}^{1},P_{t}^{1}\right)
-f_{phase}\left(  t,X_{t}^{2},P_{t}^{2}\right)  ,\medskip\\
\Delta g_{phase,t}:=g_{phase}\left(  t,X_{t}^{1},P_{t}^{1}\right)
-g_{phase}\left(  t,X_{t}^{2},P_{t}^{2}\right)  .
\end{array}
\right.
\]
The phase-space system for the difference becomes:%
\[
\left\{
\begin{array}
[c]{l}%
d\Delta X_{t}=-s_{t}d\Delta K_{t}-g_{t}J\Delta P_{t}dt+\Delta f_{t}dt+\Delta
g_{t}dW_{t}\medskip\\
d\Delta P_{t}=-g_{t}Jd\Delta K_{t}-\Gamma_{t}\Delta P_{t}dt+\Delta
f_{phase,t}dt+\Delta g_{phase,t}dW_{t}%
\end{array}
\right.
\]
Let us now construct a suitable energy functional.

Fix an arbitrary constant $\rho\geq0$ and define%
\begin{equation}
\lambda_{t}:=s_{t}\left(  s_{t}+\rho\right)  -g_{t}^{2}=\alpha_{t}\cos
\theta_{t}\left(  \alpha_{t}\cos\theta_{t}+\rho\right)  -\alpha_{t}^{2}%
\sin^{2}\theta_{t}~,\quad t\in\left[  0,T\right]  . \label{lbl}%
\end{equation}
Assume that there exists $\lambda_{0}>0$ such that $\lambda_{t}\geq\lambda
_{0}>0,$ for all $t\in\left[  0,T\right]  $. This condition is not very
restrictive for acceptable values for $\alpha$ and $\theta.$ For example, for
$\rho=0,$ $\left\vert \theta_{t}\right\vert <\pi/4.$

Define the block matrix:%
\[
Q_{t}:=\left(
\begin{array}
[c]{cc}%
\left(  s_{t}+\rho\right)  I_{2} & g_{t}J\medskip\\
-g_{t}J & s_{t}I_{2}%
\end{array}
\right)
\]
Since $\left(  g_{t}J\right)  ^{T}=g_{t}J^{T}=-g_{t}J$ it follows that $Q_{t}$
is a symmetric matrix, for all $t\in\left[  0,T\right]  $.

Define the relative energy, for $Z_{t}:=\left(  \Delta X_{t},\Delta
P_{t}\right)  ^{t},$%
\begin{equation}
\mathcal{E}_{t}:=\dfrac{1}{2}\left\langle Q_{t}Z_{t},Z_{t}\right\rangle
\label{Relative energy}%
\end{equation}
In explicit form, it becomes:%
\begin{equation}%
\begin{array}
[c]{ccl}%
\mathcal{E}_{t} & = & \dfrac{1}{2}\left\langle Q_{t}\left(
\begin{array}
[c]{c}%
\Delta X_{t}\\
\Delta P_{t}%
\end{array}
\right)  ,\left(
\begin{array}
[c]{c}%
\Delta X_{t}\\
\Delta P_{t}%
\end{array}
\right)  \right\rangle =\dfrac{1}{2}\left\langle \left(
\begin{array}
[c]{cc}%
\left(  s_{t}+\rho\right)  I_{2} & g_{t}J\medskip\\
-g_{t}J & s_{t}I_{2}%
\end{array}
\right)  \left(
\begin{array}
[c]{c}%
\Delta X_{t}\\
\Delta P_{t}%
\end{array}
\right)  ,\left(
\begin{array}
[c]{c}%
\Delta X_{t}\\
\Delta P_{t}%
\end{array}
\right)  \right\rangle \medskip\\
& = & \dfrac{1}{2}\left\langle \left(
\begin{array}
[c]{c}%
\left(  s_{t}+\rho\right)  \Delta X_{t}+g_{t}J\Delta P_{t}\medskip\\
-g_{t}J\Delta X_{t}+s_{t}\Delta P_{t}%
\end{array}
\right)  ,\left(
\begin{array}
[c]{c}%
\Delta X_{t}\\
\Delta P_{t}%
\end{array}
\right)  \right\rangle \medskip\\
& = & \dfrac{1}{2}\left[  \left(  s_{t}+\rho\right)  \left\vert \Delta
X_{t}\right\vert ^{2}+g_{t}\left(  \left\langle J\Delta P_{t},\Delta
X_{t}\right\rangle -\left\langle \Delta X_{t},J^{T}\Delta P_{t}\right\rangle
\right)  +s_{t}\left\vert \Delta P_{t}\right\vert ^{2}\right]  \medskip\\
& \overset{J^{T}=-J}{=} & \dfrac{1}{2}\left(  s_{t}+\rho\right)  \left\vert
\Delta X_{t}\right\vert ^{2}+g_{t}\left\langle J\Delta P_{t},\Delta
X_{t}\right\rangle +\dfrac{1}{2}s_{t}\left\vert \Delta P_{t}\right\vert ^{2}.
\end{array}
\label{expanded form energy}%
\end{equation}
Equivalently,%
\[%
\begin{array}
[c]{ccl}%
2\mathcal{E}_{t} & = & \left(  s_{t}+\rho\right)  \left\vert \Delta
X_{t}\right\vert ^{2}+2g_{t}\left\langle J\Delta P_{t},\Delta X_{t}%
\right\rangle +s_{t}\left\vert \Delta P_{t}\right\vert ^{2}\medskip\\
& = & \left(  s_{t}+\rho\right)  \left\vert \Delta X_{t}+\dfrac{g_{t}}%
{s_{t}+\rho}J\Delta P_{t}\right\vert ^{2}+\left(  s_{t}-\dfrac{g_{t}^{2}%
}{s_{t}+\rho}\right)  \left\vert \Delta P_{t}\right\vert ^{2}\medskip\\
& \overset{(\ref{lbl})}{=} & \left(  s_{t}+\rho\right)  \left\vert \Delta
X_{t}+\dfrac{g_{t}}{s_{t}+\rho}J\Delta P_{t}\right\vert ^{2}+\dfrac
{\lambda_{t}}{s_{t}+\rho}\left\vert \Delta P_{t}\right\vert ^{2}.
\end{array}
\]
It follows that the family of symmetric matrices $Q$ is uniformly positive
defined. Also, there exist the constants $0<c_{Q}\leq C_{Q}<+\infty$ such that%
\begin{equation}
c_{Q}\left(  \left\vert \Delta X_{t}\right\vert ^{2}+\left\vert \Delta
P_{t}\right\vert ^{2}\right)  \leq2\mathcal{E}_{t}\leq C_{Q}\left(  \left\vert
\Delta X_{t}\right\vert ^{2}+\left\vert \Delta P_{t}\right\vert ^{2}\right)
,\quad\forall t\in\left[  0,T\right]  \label{positiv. defined of Q}%
\end{equation}
From the definition of $s$ and $g$, we can assume that they are bounded and
$s_{t}\geq s_{0}>0$.

We use the notations:%
\begin{equation}
\left\{
\begin{array}
[c]{l}%
Z_{t}:=\left(
\begin{array}
[c]{c}%
\Delta X_{t}\\
\Delta P_{t}%
\end{array}
\right)  ,\quad B_{t}:=\left(
\begin{array}
[c]{c}%
s_{t}I_{2}\\
g_{t}J
\end{array}
\right)  ,\quad\Delta F_{t}:=\left(
\begin{array}
[c]{c}%
\Delta f_{t}\\
\Delta f_{phase,t}%
\end{array}
\right)  \medskip\\
\Delta\Sigma_{t}:=\left(
\begin{array}
[c]{c}%
\Delta g_{t}\\
\Delta g_{phase,t}%
\end{array}
\right)  \quad\text{and}\quad A_{t}:=\left(
\begin{array}
[c]{cc}%
0 & g_{t}J\medskip\\
0 & \Gamma_{t}%
\end{array}
\right)  .
\end{array}
\right.  \label{notations uniq}%
\end{equation}
For the difference system, the bounded variation term is%
\[
-B_{t}d\Delta K_{t}~,
\]
and, therefore, in It\^{o}'s formula for the energy $\mathcal{E}_{t}$, the
term produced by $d\Delta K_{t}$ is: $-\left\langle Q_{t}Z_{t},B_{t}d\Delta
K_{t}\right\rangle $.

Since%
\[
B_{t}^{T}Q_{t}=\left(
\begin{array}
[c]{cc}%
s_{t}I_{2} & -g_{t}J
\end{array}
\right)  \left(
\begin{array}
[c]{cc}%
\left(  s_{t}+\rho\right)  I_{2} & g_{t}J\medskip\\
-g_{t}J & s_{t}I_{2}%
\end{array}
\right)  =\left(
\begin{array}
[c]{cc}%
\lambda_{t}I_{2} & 0
\end{array}
\right)  ,
\]
it follows%
\begin{equation}
-\left\langle Q_{t}Z_{t},B_{t}d\Delta K_{t}\right\rangle =-\left\langle
B_{t}^{T}Q_{t}Z_{t},d\Delta K_{t}\right\rangle =-\left\langle \left(
\begin{array}
[c]{cc}%
\lambda_{t}I_{2} & 0
\end{array}
\right)  \left(
\begin{array}
[c]{c}%
\Delta X_{t}\\
\Delta P_{t}%
\end{array}
\right)  ,d\Delta K_{t}\right\rangle =-\lambda_{t}\left\langle \Delta
X_{t},d\Delta K_{t}\right\rangle \label{important sign}%
\end{equation}
But, since $dK_{t}^{i}\in\partial\varphi\left(  X_{t}^{i}\right)  \left(
dt\right)  $, for $i=1,2$ and for all $t\in\left[  0,T\right]  $, the
monotonicity of $\partial\varphi$ implies%
\begin{equation}%
{\displaystyle\int_{0}^{t}}
\left\langle \Delta X_{s},d\Delta K_{s}\right\rangle \geq0,\quad\text{for all
}t\in\left[  0,T\right]  . \label{monotonicity condition uniq}%
\end{equation}
Remark that, in order to obtain, from (\ref{monotonicity condition uniq}), the
contribution of the BV term with a favorable sign in It\^{o}'s formula,%
\begin{equation}
-%
{\displaystyle\int_{0}^{t}}
\lambda_{s}\left\langle \Delta X_{s},d\Delta K_{s}\right\rangle \leq0,\quad
t\in\left[  0,T\right]  \label{contrib of multiv term}%
\end{equation}
the easiest way is to choose $\lambda_{t}\equiv\lambda>0$ fixed. In other
words, we have to take $t\mapsto\rho_{t}>0$ instead of a fixed constant $\rho
$, verifying%
\[
\lambda=\lambda_{t}=s_{t}\left(  s_{t}+\rho_{t}\right)  -g_{t}^{2}%
\quad\Leftrightarrow\quad\rho_{t}=\frac{\lambda+g_{t}^{2}-s_{t}^{2}}{s_{t}%
}>0,\quad\text{for }\lambda>0\text{ big enough.}%
\]
Also, the regularity of $s$ and $g$ implies $\rho\in C^{1}\left(  \left[
0,T\right]  \right)  $ and the inequality (\ref{contrib of multiv term}) takes
place, with $\lambda$ a positive constant in front of the Riemann-Stieltjes
integral.\medskip

Supposing first that $s,g,\rho\in W^{1,\infty}\left(  0,T\right)  $ and
$\Gamma$ is bounded, we apply It\^{o}'s formula for the time dependent
relative energy, given in compact form in (\ref{Relative energy}), or, in
expanded form in (\ref{expanded form energy}):%
\[
\left(  t,x,p\right)  \mapsto u\left(  t,x,p\right)  :=\frac{1}{2}\left(
s_{t}+\rho_{t}\right)  \left\vert x\right\vert ^{2}+g_{t}\left\langle
Jp,x\right\rangle +\frac{1}{2}s_{t}\left\vert p\right\vert ^{2}.
\]
The dynamic of $Z$ is%
\begin{equation}
dZ_{t}=-B_{t}d\Delta K_{t}-A_{t}Z_{t}dt+\Delta F_{t}dt+\Delta\Sigma_{t}%
dW_{t}~. \label{dynZ}%
\end{equation}
Therefore, It\^{o}'s formula and (\ref{important sign}) give us%
\[%
\begin{array}
[c]{ccl}%
d\mathcal{E}_{t} & = & \dfrac{1}{2}\left\langle \dot{Q}_{t}Z_{t}%
,Z_{t}\right\rangle dt+\left\langle Q_{t}Z_{t},dZ_{t}\right\rangle +\dfrac
{1}{2}\mathrm{Tr}\left(  Q_{t}\Delta\Sigma_{t}\Delta\Sigma_{t}^{T}\right)
dt\medskip\\
& \overset{(\ref{dynZ})}{=} & -\lambda\left\langle \Delta X_{t},d\Delta
K_{t}\right\rangle +\left[  \dfrac{1}{2}\left\langle \dot{Q}_{t}Z_{t}%
,Z_{t}\right\rangle -\left\langle Q_{t}Z_{t},A_{t}Z_{t}\right\rangle
+\left\langle Q_{t}Z_{t},\Delta F_{t}\right\rangle \right]  dt\medskip\\
&  & +\dfrac{1}{2}\mathrm{Tr}\left(  Q_{t}\Delta\Sigma_{t}\Delta\Sigma_{t}%
^{T}\right)  dt+\left\langle Q_{t}Z_{t},\Delta\Sigma_{t}dW_{t}\right\rangle
\end{array}
\]
Under integral form, it becomes%
\begin{equation}%
\begin{array}
[c]{l}%
\mathcal{E}_{t}+\lambda%
{\displaystyle\int_{0}^{t}}
\left\langle \Delta X_{s},d\Delta K_{s}\right\rangle \medskip\\
\quad\quad=\mathcal{E}_{0}+%
{\displaystyle\int_{0}^{t}}
\left[  \dfrac{1}{2}\left\langle \dot{Q}_{s}Z_{s},Z_{s}\right\rangle
-\left\langle Q_{s}Z_{s},A_{s}Z_{s}\right\rangle +\left\langle Q_{s}%
Z_{s},\Delta F_{s}\right\rangle \right]  ds\medskip\\
\quad\quad+\dfrac{1}{2}%
{\displaystyle\int_{0}^{t}}
\mathrm{Tr}\left(  Q_{s}\Delta\Sigma_{s}\Delta\Sigma_{s}^{T}\right)  ds+%
{\displaystyle\int_{0}^{t}}
\left\langle Q_{s}Z_{s},\Delta\Sigma_{s}dW_{s}\right\rangle .
\end{array}
\label{energy in integral form}%
\end{equation}
From the boundedness of $s,g,\rho\in W^{1,\infty}\left(  0,T\right)  $ one can
obtain that there exist some positive constant $C_{Q},C_{\dot{Q}}~$and
$C_{A}~,$ which form can be explicitly computed, such that $\left\vert
\left\vert Q_{t}\right\vert \right\vert _{\mathrm{op}}\leq C_{Q},$ $\left\vert
\left\vert \dot{Q}_{t}\right\vert \right\vert _{\mathrm{op}}\leq C_{\dot{Q}}$,
and $\left\vert \left\vert A_{t}\right\vert \right\vert _{\mathrm{op}}\leq
C_{A}~,$ a.e. $t\in\left[  0,T\right]  $. It follows%
\[
\left\vert \dfrac{1}{2}\left\langle \dot{Q}_{s}Z_{s},Z_{s}\right\rangle
\right\vert \leq\dfrac{1}{2}\left\vert \left\vert \dot{Q}\right\vert
\right\vert _{\mathrm{op}}\left\vert Z_{t}\right\vert ^{2}\leq\frac{C_{\dot
{Q}}}{2}\left\vert Z_{t}\right\vert ^{2}\leq\frac{C_{\dot{Q}}}{c_{Q}%
}\mathcal{E}_{t}~,
\]
since $\left\vert Z_{t}\right\vert ^{2}=\left\vert \Delta X_{t}\right\vert
^{2}+\left\vert \Delta P_{t}\right\vert ^{2}\leq\frac{2}{c_{Q}}\mathcal{E}%
_{t}~$, according to the coercivity condition (\ref{positiv. defined of Q}).
Also,%
\[
A_{t}Z_{t}=\left(
\begin{array}
[c]{cc}%
0 & g_{t}J\medskip\\
0 & \Gamma_{t}%
\end{array}
\right)  \left(
\begin{array}
[c]{c}%
\Delta X_{t}\\
\Delta P_{t}%
\end{array}
\right)  =\left(
\begin{array}
[c]{c}%
g_{t}J\Delta P_{t}\\
\Gamma_{t}\Delta P_{t}%
\end{array}
\right)
\]
and, using again the coercivity condition (\ref{positiv. defined of Q}),%
\[
\left\vert -\left\langle Q_{t}Z_{t},A_{t}Z_{t}\right\rangle \right\vert
\leq\left\vert \left\vert Q_{t}\right\vert \right\vert _{\mathrm{op}%
}\left\vert \left\vert A_{t}\right\vert \right\vert _{\mathrm{op}}\left\vert
Z_{t}\right\vert ^{2}\leq\sup_{t\in\left[  0,T\right]  }\left\vert \left\vert
Q_{t}\right\vert \right\vert _{\mathrm{op}}C_{A}\left\vert Z_{t}\right\vert
^{2}\leq\frac{2C_{Q}^{\ast}C_{A}}{c_{Q}}\mathcal{E}_{t}~.
\]
Consequently,%
\[
\left\vert \dfrac{1}{2}\left\langle \dot{Q}_{t}Z_{t},Z_{t}\right\rangle
-\left\langle Q_{t}Z_{t},A_{t}Z_{t}\right\rangle \right\vert \leq\left(
\frac{C_{\dot{Q}}}{c_{Q}}+\frac{2C_{Q}^{\ast}C_{A}}{c_{Q}}\right)
\mathcal{E}_{t}=C\mathcal{E}_{t}~,
\]
where $C>0$ is independent of $t$.

Similarly, the Lipschitz properties of the drifts and diffusions provide%
\[
\left\vert \left\langle Q_{t}Z_{t},\Delta F_{t}\right\rangle \right\vert \leq
C\mathcal{E}_{t},\quad\left\vert \left\vert \Delta\Sigma_{t}\right\vert
\right\vert ^{2}\leq L^{2}\left(  \left\vert \Delta X_{t}\right\vert
^{2}+\left\vert \Delta P_{t}\right\vert ^{2}\right)  \leq C\mathcal{E}_{t}%
\]
and, since $Q_{t}$ is bounded, $\mathrm{Tr}\left(  Q_{t}\Delta\Sigma_{t}%
\Delta\Sigma_{t}^{T}\right)  \leq C\mathcal{E}_{t}.$

Putting all the above estimates into (\ref{energy in integral form}), and
using (\ref{contrib of multiv term}) for $\lambda=const.$, we obtain%
\begin{equation}
\mathcal{E}_{t}\leq\mathcal{E}_{0}+C%
{\displaystyle\int_{0}^{t}}
\mathcal{E}_{s}ds+M_{t}~,\quad\text{where}\quad M_{t}:=%
{\displaystyle\int_{0}^{t}}
\left\langle Q_{s}Z_{s},\Delta\Sigma_{s}dW_{s}\right\rangle ,\quad\forall
t\in\left[  0,T\right]  . \label{almost final}%
\end{equation}
For the It\^{o} integral, we apply the Burkholder--Davis--Gundy inequality and
we have, for a positive Young constant $\eta$, which we will choose later,%
\begin{equation}%
\begin{array}
[c]{ccl}%
\mathbb{E}\left(  \sup\limits_{0\leq r\leq t}\left\vert M_{r}\right\vert
\right)  & \leq & C\mathbb{E}\left[  \left(
{\displaystyle\int_{0}^{t}}
\left\vert Q_{s}Z_{s}\right\vert ^{2}\left\vert \left\vert \Delta\Sigma
_{s}\right\vert \right\vert _{F}^{2}ds\right)  ^{1/2}\right]  \medskip\\
& \leq & \eta\mathbb{E}\left(  \sup\limits_{0\leq r\leq t}\mathcal{E}%
_{r}\right)  +C_{\eta}%
{\displaystyle\int_{0}^{t}}
\mathbb{E}\left(  \sup\limits_{0\leq u\leq s}\mathcal{E}_{u}\right)  ds.
\end{array}
\label{mtg.est}%
\end{equation}
Take $\sup_{r\in\left[  0,t\right]  }$ in (\ref{almost final}) and then apply
the expectation. We obtain, by denoting $\mathcal{E}_{t}^{\ast}:=\sup_{0\leq
r\leq t}\mathcal{E}_{r}$~,%
\[%
\begin{array}
[c]{rcl}%
\mathcal{E}_{t}^{\ast}=\sup\limits_{0\leq r\leq t}\mathcal{E}_{t} & \leq &
\mathcal{E}_{0}+C\sup\limits_{0\leq r\leq t}%
{\displaystyle\int_{0}^{r}}
\mathcal{E}_{s}+\sup\limits_{0\leq r\leq t}\left\vert M_{r}\right\vert
=\mathcal{E}_{0}+C%
{\displaystyle\int_{0}^{t}}
\mathcal{E}_{s}ds+\sup\limits_{0\leq r\leq t}\left\vert M_{r}\right\vert
\medskip\\
& \leq & \mathcal{E}_{0}+C%
{\displaystyle\int_{0}^{t}}
\sup\limits_{0\leq u\leq s}\mathcal{E}_{u}ds+\sup\limits_{0\leq r\leq
t}\left\vert M_{r}\right\vert \medskip\\
& \leq & \mathcal{E}_{0}+C%
{\displaystyle\int_{0}^{t}}
\mathcal{E}_{s}^{\ast}ds+\sup\limits_{0\leq r\leq t}\left\vert M_{r}%
\right\vert .\quad\quad\text{Apply }\mathbb{E}\text{ and it follows}\medskip\\
\mathbb{E}\left(  \mathcal{E}_{t}^{\ast}\right)  & \leq & \mathbb{E}\left(
\mathcal{E}_{0}\right)  +C%
{\displaystyle\int_{0}^{t}}
\mathbb{E}\left(  \mathcal{E}_{s}^{\ast}\right)  ds+\mathbb{E}\left(
\sup\limits_{0\leq r\leq t}\left\vert M_{r}\right\vert \right)  \medskip\\
& \overset{(\ref{mtg.est})}{\leq} & \mathbb{E}\left(  \mathcal{E}_{0}\right)
+\left(  C+C_{\eta}\right)
{\displaystyle\int_{0}^{t}}
\mathbb{E}\left(  \mathcal{E}_{s}^{\ast}\right)  ds+\eta\mathbb{E}\left(
\mathcal{E}_{t}^{\ast}\right)
\end{array}
\]
Finally,%
\[
\left(  1-\eta\right)  \mathbb{E}\left(  \mathcal{E}_{t}^{\ast}\right)
\leq\mathbb{E}\left(  \mathcal{E}_{0}\right)  +\left(  C+C_{\eta}\right)
{\displaystyle\int_{0}^{t}}
\mathbb{E}\left(  \mathcal{E}_{s}^{\ast}\right)  ds,\quad\forall t\in\left[
0,T\right]  .
\]
Let us select the absorption constant $\eta=1/2$ and apply Gronwall's
inequality. Hence%
\[
\mathbb{E}\left(  \sup_{0\leq r\leq t}\mathcal{E}_{r}\right)  \leq
C_{T}\mathbb{E}\left(  \mathcal{E}_{0}\right)
\]
In order to obtain the uniqueness, take $\Delta X_{0}=\Delta P_{0}=0.$ Then
$\mathcal{E}_{0}=0$ and it follows that $\mathcal{E}_{r}=0,$ for all
$r\in\left[  0,t\right]  ,$ $\mathbb{P-}$a.s. The coercivity of the energy
implies%
\[
\Delta X_{r}=0,\quad\Delta P_{r}=0,\quad\forall r\in\left[  0,t\right]
,\quad\mathbb{P-}\text{a.s.}%
\]
But $t\in\left[  0,T\right]  $ is arbitrary, which gives $\Delta X_{t}=0,$
$\Delta P_{t}=0,$ $\forall t\in\left[  0,T\right]  ,$ $\mathbb{P-}$a.s.
Consider now the first equation from the difference Phase-Space system. Since
all the components except $s_{t}d\Delta K_{t}$ are equal with zero and $s_{t}$
takes only positive values, it follows $d\Delta K_{t}\equiv0$. Finally,
$\Delta K_{t}=0,$ $\forall t\in\left[  0,T\right]  ,$ $\mathbb{P-}$a.s. and
the pathwise uniqueness follows:%
\[
\left(  X^{1},P^{1},K^{1}\right)  =\left(  X^{2},P^{2},K^{2}\right)
,\quad\mathbb{P-}\text{a.s.}%
\]
The uniqueness proof is now complete.\hfill
\end{proof}

\end{document}